\documentclass[fontsize=10pt]{scrreprt}

\usepackage[square,numbers,nonamebreak]{natbib}
\usepackage{setspace}

\usepackage[english]{babel}

\usepackage{chngcntr}
\usepackage[hang]{footmisc}
 at 11truept
\font\ssc=pplrc9d at 11 truept
\newcommand\qedbox{$\rlap{$\sqcap$}\sqcup$}

\usepackage[hmarginratio=1:1, bottom=2cm, top=2.5cm]{geometry}

\usepackage[automark]{scrlayer-scrpage}
\cohead{}
\let\ceheadL\cehead
\renewcommand\cehead[1]{
\ceheadL{\textnormal{#1}}
}

\lefoot{}
\rofoot{}

\usepackage{graphicx}
\usepackage[usenames,dvipsnames,svgnames,table]{xcolor}
\definecolor{Maroon}{cmyk}{0, 0.87, 0.68, 0.32}
\definecolor{RoyalBlue2}{cmyk}{80,100,0,0.1}

\newcommand\auths[1]{\large \textsc{\textcolor{Maroon}{#1}}\setstretch{1.2}}
\newcommand\titl[1]{\center \linespread{1.1}\color{RoyalBlue2}\Large\textbf{ #1}\color{black}\bigskip} 
\renewcommand\abstract[1]{
\begin{center}
{\textbf{Abstract}}
\end{center}
{
\linespread{1.1}\fontsize{9pt}{-10pt}\selectfont #1}}

\usepackage[utf8]{inputenc}
\usepackage[sc,osf]{mathpazo}
\usepackage[euler-digits]{eulervm}
\usepackage[T1]{fontenc}
\usepackage{mathtools}

\DeclareSymbolFont{operators}{\encodingdefault}{ppl}{m}{n}
\DeclareMathAlphabet{\mathbf}{\encodingdefault}{ppl}{bx}{n}
\DeclareMathAlphabet{\mathit}{\encodingdefault}{ppl}{m}{it}

\usepackage{amsfonts}
\usepackage{amsmath}
\usepackage{ragged2e}
\usepackage{titlesec}
\usepackage{amssymb}
\usepackage{float}
\usepackage{algorithm}
\usepackage{algorithmic}
\usepackage{placeins}

\renewcommand{\thesection}{\arabic{section}}
 
\titleformat{\section}{\medskip\bigskip\normalfont\Large\bf}{\thesection}{0.5em}{}
\titleformat{\subsection}{\smallskip\bigskip\normalfont\large\bf}{\thesubsection}{0.5em}{}

\usepackage{amsthm}
\newtheoremstyle{dotless}{}{}{\itshape}{}{\bfseries}{}{1em}{}

\newtheorem*{theo*}{Theorem}
\newtheorem*{rem*}{Remark}
\theoremstyle{dotless}
\newtheorem{theo}{Theorem}

\newtheorem{lem}[theo]{Lemma}
\newtheorem{defi}[theo]{Definition}
\newtheorem{cor}[theo]{Corollary}
\newtheorem{rem}[theo]{Remark}
\newtheorem{ex}[theo]{Example}

\renewenvironment{proof}{\smallbreak\noindent {\sc Proof \;---\;}}{\hfill\qedbox}

\counterwithout{theo}{section}
\numberwithin{theo}{section}

\DeclareOldFontCommand{\rm}{\normalfont\rmfamily}{\mathrm}
\DeclareOldFontCommand{\sf}{\normalfont\sffamily}{\mathsf}
\DeclareOldFontCommand{\tt}{\normalfont\ttfamily}{\mathtt}
\DeclareOldFontCommand{\bf}{\normalfont\bfseries}{\mathbf}
\DeclareOldFontCommand{\it}{\normalfont\itshape}{\mathit}
\DeclareOldFontCommand{\sl}{\normalfont\slshape}{\@nomath\sl}
\DeclareOldFontCommand{\sc}{\normalfont\scshape}{\@nomath\sc}
\newcommand{\bigcdot}{\boldsymbol{\cdot}}

\usepackage{comment}

\usepackage{soul}
\DeclareSymbolFont{newfont}{OML}{cmm}{m}{it}%
\DeclareMathSymbol{\Varrho}{3}{newfont}{37}

\begin{document}

\titl{Finite skew braces of square-free order and supersolubility}\footnote{All authors are members of the non-profit association ``Advances in Group Theory and Applications'' (www.advgrouptheory.com). The third and fifth authors are supported by GNSAGA (INdAM). The third author has been supported by a research visiting grant issued by the Istituto Nazionale di Alta Matematica (INdAM).}

\auths{A. Ballester-Bolinches -- R. Esteban-Romero -- M. Ferrara\\ V.~P\'erez-Calabuig -- M. Trombetti}

\thispagestyle{empty}
\justify\noindent
\setstretch{0.3}
\abstract{The aim of this paper is to study {\it supersoluble} skew braces, a class of skew braces that encompasses all finite skew braces of square-free order. It turns out that finite supersoluble skew braces have Sylow towers, and that in an arbitrary supersoluble skew brace $B$ many relevant skew brace-theoretical properties are easier to identify: for example, a centrally nilpotent ideal of $B$ is $B$-centrally nilpotent, a fact that simplifies the computational search for the Fitting ideal; also,~$B$ has finite multipermutational level if and only if~$(B,+)$ is nilpotent.  

Given a finite presentation of the structure skew brace $G(X,r)$ associated with a finite non-degenerate solution of the Yang--Baxter Equation (YBE), there is an algorithm that decides if $G(X,r)$ is supersoluble or not. Moreover, supersoluble skew braces are examples of almost polycyclic skew braces, so they give rise to solutions of the YBE on which one can algorithmically work on.
}

\setstretch{2.1}
\noindent
{\fontsize{10pt}{-10pt}\selectfont {\it Mathematics Subject Classification \textnormal(2020\textnormal)}: 16T25, 03D40, 20F10, 20F16}\\[-0.8cm]

\noindent 
\fontsize{10pt}{-10pt}\selectfont  {\it Keywords}: 
skew brace of square-free order; supersoluble skew brace; Yang--Baxter Equation\\[-0.8cm]

\setstretch{1.1}
\fontsize{11pt}{12pt}\selectfont

\bigskip\bigskip\bigskip

\section{Introduction}

A {\it skew \textnormal(left\textnormal) brace} is a set $B$ endowed with two binary operations $+$ and $\bigcdot$ such that~$(B,+)$ and $(B,\bigcdot)$ are groups, and the {\it skew distributive law} holds, that is, $$x\bigcdot (y + z) = (x\bigcdot y)-x+(x\bigcdot z)$$ for $x, y, z\in B$. This structure has been introduced in \cite{GuarnieriVendramin17} as a generalisation of the left braces introduced by Rump in \cite{Rump07} (these are the skew left braces in which the operation $+$ is commutative). Skew left braces have been devised with the aim of attacking the problem of finding all  set-theoretical non-degenerate solutions ({\it solutions}, for short) of the Yang--Baxter~Equa\-tion (YBE), a consistency equation that plays a relevant role in quantum statistical mechanics, in the foundation of quantum groups, and that provides a multidisciplinary approach from a wide variety of  areas such as Hopf algebras, knot theory and braid theory among others (see \cite{plus2},\cite{plus1},\cite{baxter},\cite{drinfeld90},\cite{Faddev},\cite{Gateva-Ivanova18-advmath},\cite{yang}). In fact, to every solution of the~YBE, one can associate a skew left brace (which is usually called the {\it structure} skew left brace of the solution), and conversely to every skew left brace one can associate a solution of the YBE (so properties of solutions can be translated in terms of the structure skew left braces and vice-versa). 

In this way, there is a wide consensus on the need of delving into the study of skew left braces not only to enrich the theory of this algebraic structure but also to provide proper frameworks that work out well for translating and classifying properties of solutions of the YBE.
%(see for example \cite{5},\cite{advmathcedo},\cite{Schur},\cite{periodici}).
Significant examples of such a fruitful approach can be found in the analysis of two of the most studied properties of solutions since their introduction in \cite{EtingofShedlerSoloviev99}: \emph{multipermutation solutions}, i.e. those solutions that can be retracted into the trivial solution over a singleton after finitely many identification steps, and \emph{indecomposable solutions}, or those solutions that can not be decomposed in a disjoint union of two proper solutions. Concretely, nilpotency of skew left braces has been introduced to deal with the former ones (see~\cite{5,Cedo,GI-Cameron12, periodici}, among others), and simplicity of skew left braces play a key role in the study of indecomposable solutions (see \cite{CCP19, CedoOkninski21, advmathcedo,VendraminExtensions}). Moreover, in \cite{solubleref} solubility of skew left braces has been  studied as an opposite class of simplicity that allows multidecomposability of solutions. These significant examples show that the classification of (finite) skew left braces undoubtedly provides a powerful tool to classify the solutions of the YBE (to some extent, at least).

%There is no doubt this approach has been very useful in studying the solutions of the YBE (see for example \cite{5},\cite{advmathcedo},\cite{Schur},\cite{periodici}), and we can even give the inexperienced reader a feeling of this fact by noting that there are almost five million involutive solutions of size $10$, while there are just $6$ skew left braces of order $10$; we also note that for every finite skew left brace $B$, there exists a finite solution of the YBE such that the quotient of the structure skew left brace with respect to its socle is isomorphic to $B$ (see \cite{bachiller}). Thus, we see why the classification of (finite) skew left braces undoubtedly provides a powerful tool to classify the solutions of the YBE (to some extent, at least).

The aim of this paper is to give a further contribution to the classification of (finite) skew left braces by introducing a class of skew left braces encompassing, among others, all finite skew left braces of square-free order (see~The\-o\-rem~\ref{squarefreeorder} and~Co\-rol\-la\-ry~\ref{corsquarefreeorder}): the class of {\it supersoluble} skew left braces. One of the main properties of a finite supersoluble skew left brace $B$ is that it has a {\it Sylow tower}, that is, if $P_1,\ldots,P_n$ are a set of representatives for the Sylow subgroups of $(B,+)$, then there is a permutation \hbox{$\sigma\in\operatorname{Sym}(n)$} such that $P_{\sigma(1)}$, $P_{\sigma(1)}P_{\sigma(2)}$, \ldots, $P_{\sigma(1)}P_{\sigma(2)}\ldots P_{\sigma(n)}$ are ideals of $B$ (see~The\-o\-rem~\ref{theoincredible} and Corollary \ref{psubbrace}). This a very nice property which is also shared by the structure skew left braces associated with an indecomposable involutive solution of square-free order (see \cite{advmathcedo}). Actually, in case of a supersoluble skew left brace~$B$, we know much more: in fact, $\sigma$ orders the $P_i$'s according to decreasing magnitude of the primes they are associated with; in particular, the set of all odd-order elements of~$(B,+)$ is an ideal of~$B$ (see Corollary \ref{psubbrace}).

\medskip

As noted in \cite{periodici}, computer experiments are making more and more clear that almost every finite solution of the YBE is a multipermutation solution.
%, that is, it can be retracted into the trivial solution over a singleton after finitely many identification steps. 
Since a solution is multipermutation if and only if the structure skew left brace has finite multipermutational level (see \cite{55}, Theorem 4.13), it is relevant to find  computationally cheap ways to understand if a skew left brace has finite multipermutational level or not. Here we prove that for a supersoluble skew left brace, having finite multipermutational level is equivalent to having a nilpotent additive group (see~The\-o\-rem~\ref{rnilsupersoluble}). 

A further nice feature of supersoluble skew left braces is that they have a largest centrally nilpotent ideal (the {\it Fitting ideal}), an ideal that does not exist even in arbitrary finite skew left braces (see \cite{tutti23}) --- in fact, a centrally nilpotent ideal $I$ of a supersoluble skew left brace $B$ has a finite chain of ideals whose factors are central in $I$ (see~The\-o\-rem~\ref{theobcentrally}). In the supersoluble environment, the Fitting ideal has finite index (see~The\-o\-rem~\ref{fittingfiniteindex}), and on some occasions its index is even a power of $2$ (see Theorem \ref{fittingpowerof2}).

In connection with nilpotency properties of supersoluble skew left braces, we also cite Theorem \ref{leftnilptheo}, which deals with left-nilpotency, and Theorem \ref{schurtheor}, which deals with an extension of \cite{Schur}, Theorem 5.4.

\medskip

Many good features of supersoluble skew left braces are inherited from those of the almost polycyclic skew left braces that are studied in \cite{tutti23-2}. Thus, supersoluble skew left braces provide a further class of solutions of the YBE on which we can algorithmically work on (given a finite presentation of the supersoluble skew left brace). But this also means that supersoluble skew left braces are residually finite, although in this context we can say something more (see Theorem \ref{resfinite}). Moreover, we note that a maximal sub-skew left brace has prime index (see Theorem \ref{maximalsubbrace}) and a sub-skew left brace of index $2$ is an ideal (see The\-o\-rem~\ref{finitecasesuperso}).

\medskip

It is proved in \cite{tutti23-2} that many natural properties of an almost polycyclic skew left brace can be checked by only looking at its finite homomorphic images. In this context, the property of being supersoluble makes no exception (see Theorem \ref{finitehomsupers}). In turns, this provides us with an algorithm that let us decide if an almost polycyclic skew left brace (and in particular the structure skew left brace of a finite solution of the~YBE) is supersoluble or not (see Theorem \ref{decide}), given only a finite presentation of it.

\medskip

Finally, it should be remarked that some of our results are proved in (and are generalized to) the larger classes of skew left braces given by the locally supersoluble skew left braces and the hypercyclic skew left braces (see Section \ref{sectls}). Among other things, in the final section we focus on chief factors, on the relation between the~Frattini and the~Fitting ideals, and on the role played by the countable sub-skew left braces.

\section{Preliminaries}

\emph{From now on, the word ``brace'' means ``skew left brace''}.

\medskip

The aim of this section is to fix notation, and to give  background concepts and results that we need in the paper. We refer the reader to \cite{tutti23-2} and \cite{GuarnieriVendramin17} for either undefined concepts or more detailed descriptions of the given ones. 

Let $(B,+,\bigcdot)$ be a brace. If $(B,+)$ is abelian (resp. satisfies some property $\mathfrak{X}$), we say that $B$ is a brace of {\it abelian type} (resp. {\it of $\mathfrak{X}$ type}). The symbol $0$ denotes the common identity element of~$(B,+)$ and $(B,\bigcdot)$, while $b^{-1}$ and $-b$ denote the multiplicative and additive inverses of~$b$, respectively. We use juxtaposition for the product of elements, and, if $n\in\mathbb{Z}$ and~\hbox{$b\in B$,} the following notation is adopted: $$nb=\begin{cases}
\underbrace{b+\ldots+b}_{\textnormal{$n$ times}} & \textnormal{if $n\geq0$}\\[0.6cm]
$\,$\;0 & \textnormal{if $n=0$}\\[0.2cm]
-((-n)b) & \textnormal{if $n<0$}
\end{cases}\quad\textnormal{and}\quad 
b^n=\begin{cases}
\underbrace{b\bigcdot\ldots\bigcdot b}_{\textnormal{$n$ times}} & \textnormal{if $n\geq0$}\\[0.6cm]
$\,$\;0 & \textnormal{if $n=0$}\\[0.2cm]
(b^{-n})^{-1} & \textnormal{if $n<0$}
\end{cases}
$$ The homomorphism $$\lambda\colon b\in (B,\bigcdot)\mapsto \lambda_b \in \operatorname{Aut}(B,+),$$ defined through the position $\lambda_a(b) = -a + ab$ for every $a,b\in B$, relates the groups $(B,\bigcdot)$ and $(B,+)$, while the \emph{star operation} measures the difference between both operations: $a\ast b = -a +ab -b = \lambda_a(b) - b$ for every $a,b\in B$. As usual,  addition follows product in the order of operations, and the star product comes first in the order of operations. The following properties are well-known: 
\[ \begin{array}{c}\label{dacitare}
(ab) \ast c  =  a \ast (b\ast c) + b\ast c + a \ast c,\\[0.2cm]
ab  =  a + a \ast b + b,\\[0.2cm]
a \ast (b+c)  = a \ast b + b + a \ast c - b, \qquad \text{for all $a,b,c\in B$.} \tag{$\dagger$}
\end{array} \] If $X$ and $Y$ are subsets of $B$, then $X \ast Y$ is the subgroup generated in $(B,+)$ by all elements $x\ast y$, for all $x\in X$ and $y\in Y$. A brace is \emph{trivial} if  both operations coincide, or equivalently, if $B \ast B = 0$. \emph{Abelian} braces are trivial braces whose additive group is abelian.

On some occasions, we need to emphasize that a subset $S$ of $B$ is considered as a subset of $(B,+)$ or $(B,\bigcdot)$. In order to do this, we add a $+$ symbol or a $\bigcdot$ symbol near the set $S$; so if $X$ and~$Y$ are subbraces of $B$, then $[X,Y]_+$ is the usual commutator subgroup in~$(B,+)$, while, if $n\in\mathbb{N}$, then $X^{n,\bigcdot}$ is the subgroup of $(B,\bigcdot)$ generated by all elements of type $x^n$ with $x\in X$.

A map between braces that preserves multiplications and sums is a \emph{homomorphism}. Two braces are \emph{isomorphic} if there exists a bijective homomorphism between them. In describing the substructural framework of an arbitrary brace $B$, we employ the following notation.
\begin{itemize}
\item A \emph{subbrace} $S$ of $B$ is a subset of $B$ that is a subgroup of both $(B,+)$ and $(B,\bigcdot)$. In this case, we write $S \leq B$. 
\item A \emph{left-ideal} $L$ of $B$ is a subbrace of $B$ such that $\lambda_b(L)\leq L$ for all $b\in B$. Moreover,~$L$ is a \emph{strong left-ideal} if $L$ is also a normal subgroup of $(B,+)$.
\item An \emph{ideal} $I$ of $B$ is a  strong left-ideal that is also a normal subgroup of~$(B,\bigcdot)$. In this case, we write $I \unlhd B$.
%\item A \emph{subideal} $C$ of $B$ is a subbraces $C$ of $B$ for which there is a finite chain of the following type $$C=C_0\trianglelefteq C_1\trianglelefteq\ldots \trianglelefteq C_n=B.$$
\end{itemize}
Ideals of braces allow us to consider quotients: given an ideal $I$ of $B$, $B/I$ has a natural brace structure and we have that $bI = b+I$ for each $b\in B$.  Note also that if $L$ is a left-ideal of~$B$ and $I\trianglelefteq B$, then~$I\ast L$ is a left-ideal of~$B$, and $I\ast B$ is an ideal of $B$; in particular, $B\ast B\trianglelefteq B$. The following \[
\operatorname{Soc}(B)  = \operatorname{Ker}\lambda \cap \operatorname{Z}(B,+)\quad\textnormal{and}\quad
\zeta(B)=\operatorname{Soc}(B)\cap Z(B,\bigcdot),\] (here, $Z(B,+)$ and $Z(B,\bigcdot)$ denote the centres of $(B,+)$ and $(B,\bigcdot)$, respectively) are further relevant examples of ideals of $B$ playing key roles in the study of nilpotency of braces (see~\cite{5} or~\cite{Cedo}).

\medskip

If $E$ is any subset of the brace $B$, then~$\langle E\rangle$ denotes the \emph{subbrace generated} by~$E$ in~$B$, i.e. the smallest subbrace of $B$ containing $E$ with respect to the inclusion. Moreover, $E$ is a {\it system of generators} for $\langle E\rangle$, and if $E$ is finite (of order $n$), we say that $\langle E\rangle$ is {\it finitely generated} (or {\it $n$-generator}). Also, $E^B$ denotes the \emph{ideal generated} by~$E$ in~$B$ (this is the smallest ideal of $B$ containing~$E$ with respect to the inclusion); if $E=\{a\}$ is a singleton, we write  $a^B$ instead of $\{a\}^B$. In a similar fashion, if $C\leq B$, then $C_B$ denotes the largest ideal of $B$ contained in $C$. In particular, $C\trianglelefteq B$ if and only if $C=C_B=C^B$.

\medskip

Braces can be defined from bijectives $1$-cocycles associated with actions of groups. Let $(C,\bigcdot)$ and $(B,+)$ be groups of the same order. Assume that $C$ acts on $B$ by means of a homomorphism $\lambda\colon c\in (C,\bigcdot) \mapsto \lambda_c\in \operatorname{Aut}(B,+)$. A \emph{bijective $1$-cocycle} associated with~$\lambda$ is a bijective map $\delta \colon C \rightarrow B$ such that $\delta(c_1c_2)= \delta(c_1) + \lambda_{c_1}\big(\delta(c_2)\big)$. In the previous situation we can consider the semidirect product $G = B\rtimes C$ defined by the action~$\lambda$ and written in multiplicative notation for the sake of uniformity. Following~\cite{BallesterEsteban22}, trifactorisations of the group $G$ give rise to bijective $1$-cocycles: if there exists~\hbox{$D \leq G$} such that $D\cap C = B \cap D = 1$ and $DC = BD = G$, then there exists a bijective $1$-cocycle $\delta \colon C \rightarrow B$ given by $D = \{\delta(c)c: c\in C\}$. At this point, observe that~$\delta(c)$ must be translated in additive notation. Then, $(B,+)$ admits a brace structure by means of~\hbox{$ab := \delta\big(\delta^{-1}(a)\delta^{-1}(b)\big)$,} for every $a,b\in B$ (see \cite[Proposition 1.11]{GuarnieriVendramin17}), for example). We shall repeatedly use this method to construct finite braces in the detailed examples of this paper.
\medskip

We now briefly recall the nilpotency concepts we need, and some of their properties; we refer the interested reader to \cite{tutti23},\cite{paper},\cite{bonatto},\cite{55},\cite{Cedo},\cite{periodici},\cite{Smok} for more information on the subject. Let $B$ be a brace, and let $C,D$ be subsets of $B$. Set $R_0(C;D)=C$, $L_0(C;D)=D$, and recursively define $$R_{n+1}(C;D)=R_{n}(C;D)\ast D\quad\textnormal{and}\quad L_{n+1}(C;D)=C\ast L_n(C;D)$$ for all $n\in\mathbb{N}$. Then $B$ is {\it right-nilpotent} (resp.~{\it left-nilpotent}) if $R_m(B;B)=\{0\}$ (resp. $L_m(B;B)=\{0\}$) for some $m\in\mathbb{N}$ (note that for each non-negative integer $i$, $R_i(B;B)$ is always an ideal of $B$, while $L_i(B;B)$ is a left-ideal of $B$). In order to deal with right-nilpotency, we need to define the {\it upper socle series} of a brace $B$. Let~\hbox{$\operatorname{Soc}_0(B)\!=\!\{0\}$} and recursively define $\operatorname{Soc}_{\alpha+1}(B)$ as $\operatorname{Soc}_{\alpha+1}(B)/\operatorname{Soc}_\alpha(B)=\operatorname{Soc}\big(B/\operatorname{Soc}_\alpha(B)\big)$ for any ordinal number $\alpha$; also, put $\operatorname{Soc}_\lambda(B)=\bigcup_{\alpha<\lambda}\operatorname{Soc}_\alpha(B)$ for any limit ordinal $\lambda$. We say that $B$ has {\it finite multipermutational level} if $B=\operatorname{Soc}_n(B)$ for some non-negative integer~$n$; in this case $B$ is right-nilpotent, and it turns out that the converse holds if~$B$ is of nilpotent type (see \cite{Cedo}, Lemma 2.16). Of course, the upper socle series stops at some point, and we let $\overline{\operatorname{Soc}}(B)$ be the last term of the upper socle series of~$B$. The link between right-nilpotency and the solutions of the YBE is marvelously expressed by the fact that a solution is multipermutation if and only if the associated structure brace has finite multipermutation level (see \cite{55}, Theorem 4.13). Concerning left-nilpotency, we only recall (see \cite{Cedo}, Theorem 4.8) that a finite brace $B$ of nilpotent type is left-nilpotent if and only if $(B,\bigcdot)$ is nilpotent.

The strongest nilpotency concept is central nilpotency. This is defined by using the {\it upper central series} $\{\zeta_\alpha(B)\}_{\alpha\in\operatorname{Ord}}$ of a brace~$B$, which can be defined similarly to the upper socle series but replacing the operator $\operatorname{Soc}(-)$ with the operator $\zeta(-)$. A brace~$B$ is {\it hypercentral} if~\hbox{$B=\zeta_\mu(B)$} for some ordinal $\mu$, and is {\it centrally nilpotent} if~\hbox{$B=\zeta_m(B)$} for some non-negative integer $m$. Obviously, central nilpotency implies both left- and right-nilpotency, and, conversely, left- and right-nilpotency imply central nilpotency if the brace is of nilpotent type (see~Co\-rol\-la\-ry 2.15 of \cite{periodici}).  Central nilpotency can also be defined through the {\it lower central series} of $B$. This is a descending chain of ideals of~$B$ that is recursively defined as follows: $\Gamma_1(B)=B$ and $$\Gamma_{n+1}(B) = \langle \Gamma_n(B) \ast B, B \ast \Gamma_n(B),[\Gamma_n(B), B]_+\rangle_+$$ for every $n\in\mathbb{N}$. Then $B=\zeta_m(B)$ if and only if $\Gamma_{m+1}(B)=\{0\}$. Although central nilpotency is the strongest nilpotency concept for braces, it has been shown in \cite{tutti23} that the sum of two centrally nilpotent ideals need not be centrally nilpotent, even in finite braces: the problem is that a centrally nilpotent ideal $I$ of a brace $B$ need not have a finite chain of ideals of~$B$ that are central in $I$. To avoid this problem, we gave the following definition in \cite{tutti23}: an ideal $I$ of a brace $B$ is {\it $B$-centrally nilpotent} if there exists a finite chain of ideals of~$B$ $$\{0\}=I_0\leq I_1\leq\ldots\leq I_n=I$$ with {\it $I$-central} factors, that is, such that $I_{i+1}/I_i\leq\zeta(I/I_i)$ for all $0\leq i<n$; clearly, every abelian ideal is $B$-centrally nilpotent. Also in case of $B$-central nilpotency there is a standard ascending chain of ideals of $B$. This is called the {\it upper $B$-central series} of $I$ and is defined as follows: put $\zeta_0(I)_B=\{0\}$, $\zeta_\lambda(I)_B=\bigcup_{\alpha<\lambda}\zeta_\alpha(I)_B$ for every limit ordinal~$\lambda$, and $\zeta_{\alpha+1}(I)_B/\zeta_{\alpha}(I)_B$ to be the largest ideal of $B/\zeta_{\alpha}(I)_B$ contained in~$\zeta\big(I/\zeta_\alpha(I)_B\big)$ for every ordinal number~$\alpha$; in particular, $\zeta(I)_B=\zeta_1(I)_B$ is the largest ideal of~$B$ contained in~$\zeta(I)$. Of course, $I$ is~\hbox{$B$-cen}\-tral\-ly nilpotent if and only if~\hbox{$\zeta_n(I)_B=I$} for some $n\in\mathbb{N}$. Now, if~$I$ and $J$ are $B$-centrally nilpotent ideals of~$B$, then $I+J$ is~\hbox{$B$-cen}\-tral\-ly nilpotent (see~\cite{tutti23},~The\-o\-rem~5.3); the sum of all~\hbox{$B$-centrally} nilpotent ideals of~$B$ is the {\it Fitting ideal}~$\operatorname{Fit}(B)$ of $B$. We recall also that if $I=\zeta_\mu(I)_B$ for some arbitrary ordinal number~$\mu$, then $I$ is said to be {\it $B$-hypercentral}; thus, $I$ is $B$-hypercentral if and only if it has an ascending chain of ideals of $B$ with $I$-central factors.  Of course, the upper central series stops at some point, and we let $\overline{\zeta}(B)$ be the last term of the upper central series of~$B$.

We also briefly recall that a brace $B$ is {\it soluble} if it has a finite chain of ideals $$\{0\}=I_0\leq I_1\leq\ldots\leq I_n=B$$ whose factors are abelian (as braces). Our interest in soluble braces essentially is in the fact that, as we shall see, every supersoluble brace is in fact soluble. For more details on this concept and its relationship with the solutions of the Yang--Baxter Equation, see~\cite{solubleref}.

All previous concepts can be localized as follows. If $\mathfrak{X}$ is any brace-theoretical property, we say that a brace $B$ is {\it locally $\mathfrak{X}$} if every finitely generated subbrace of $B$ satisfies $\mathfrak{X}$. Thus, for example, a locally centrally-nilpotent brace is a brace whose finitely generated subbraces are centrally nilpotent. It has been shown in \cite{Tr23} and \cite{tutti23} that a hypercentral brace $B$ is locally centrally-nilpotent and that the set of all periodic elements of $(B,+)$ is an ideal of $B$ coinciding with the set of all periodic elements of~$(B,\bigcdot)$.

 \medskip

Finally, we briefly recall the definition of presentation of a brace and some related concepts. We refer the reader to ~\cite{tutti23-2} and~\cite{Tr23} for more details; in particular, see~\cite{tutti23-2} for an elementary construction of a free brace inspired by~\cite{orza}. Let $B$ be a brace. A {\it presentation} of $B$ is an exact sequence of braces \[
\begin{array}{c}\label{exact}
0\rightarrow R\rightarrow F\xrightarrow{\theta} B\rightarrow 0,\tag{$\star$}
\end{array}
\] where $F$ is a free brace over some set $X$. The exact sequence \eqref{exact} is a {\it finite presentation} of~$B$  if $X$ is finite and there are finitely many elements $\rho_1,\ldots,\rho_n$ of $F$ generating $\operatorname{Ker}(\theta)$ as an ideal of~$F$; we also say that $B$ is {\it finitely presented} by the generators $a_1,\ldots,a_m$ subject to the relations $\rho_1=\ldots=\rho_n=1$. 

Let $\mathcal{S}=\{x_i\}_{i\in I}$ be symbols. A {\it $b$-word} with respect to~$\mathcal{S}$ is a sequence of symbols recursively defined as follows: the empty sequence is a $b$-word, and such is the~\hbox{$1$-element} sequence $x_i$ for each $i\in I$; if we have two $b$-words~$w_1$ and~$w_2$, then the sequences \hbox{$w_1\bigcdot w_2$,} $w_1+w_2$, $-w_1$, $w_1^{-1}$ are $b$-words.  If $b_1,\ldots, b_n,\ldots$ are elements of~$B$, then we may naturally evaluate any $b$-word in~$B$ by replacing $x_i$ by $b_i$.

\section{Supersoluble braces}\label{sectsuper}

Our definition of supersoluble brace is inspired by \cite{fuster?}, by the definition of almost polycyclic brace given in \cite{tutti23-2}, and by the well-known concept of supersoluble group. 

\begin{defi}\label{defsuper}
{\rm A brace $B$ is said to be {\it supersoluble} if there is a finite chain of ideals of~$B$ $$\{0\}=I_0\leq I_1\leq\ldots\leq I_n=B$$ such that, for all $0\leq i<n$, either $(I_{i+1}/I_i,+)$ is infinite cyclic and \hbox{$I_{i+1}/I_i\leq\operatorname{Soc}(B/I_i)$,} or $I_{i+1}/I_i$ has prime order.}
\end{defi}

\begin{rem}
{\rm Since every brace of prime order is abelian, we have that every supersoluble brace is soluble.}
\end{rem}

\begin{rem}
{\rm As we shall later see, this definition can somehow be weakened (see~The\-o\-rem~\ref{weaksupersoluble}), but for simplicity's sake we have preferred to start with this apparently stronger definition, then derive some of the main results, and then show that this definition is equivalent to an apparently weaker one.}
\end{rem}

\begin{rem}\label{remchepoiserve}
{\rm In connection with Definition \ref{defsuper}, it could be worthwhile to recall that~Pro\-po\-sition~4.14 of \cite{tutti23} states that if a brace has cyclic multiplicative and additive groups, then it is contains an element that additively and multiplicatively generates the whole brace.}
\end{rem}

\medskip

Of course, supersolubility (as essentially all other properties we deal with in this paper) is preserved with respect to the formation of subbraces and quotients. Also, the following result is clear.

\begin{lem}\label{supersolublebraceisgroup}
Let $B$ be a brace, and let $G=(B,+)\rtimes_\lambda(B,\bigcdot)$. If $B$ is supersoluble, then $G$  is supersoluble. In particular, $(B,+)$ and $(B,\bigcdot)$ are supersoluble groups.
\end{lem}
\begin{proof}
Let $$\{0\}=I_0\leq I=I_1\leq\ldots\leq I_n=B$$ be a chain of ideals of $B$ such that, for all $0\leq i<n$, either $(I_{i+1}/I_i,+)$ is infinite cyclic and \hbox{$I_{i+1}/I_i\leq\operatorname{Soc}(B/I_i)$,} or $I_{i+1}/I_i$ has prime order. Then $X=(I,+)$ and $Y=(I,+)\rtimes_\lambda(I,\bigcdot)$ are normal subgroups of $G$ such that both $X$ and $Y/X$ are cyclic groups. Since $G/Y\simeq (B/I,+)\rtimes_\lambda(B/I,\bigcdot)$, we are done by induction on $n$.
\end{proof}

\medskip

The converse of Lemma \ref{supersolublebraceisgroup} does not hold. In fact, we shall soon see that every finite supersoluble brace of prime power order is centrally nilpotent (see~The\-o\-rem~\ref{pbracecentrallynilp}), but there are plenty of examples of non-centrally-nilpotent braces of prime power order. In any case, the fact that many group structures connected with a supersoluble brace are in fact supersoluble makes it possible to exploit some of the good properties of supersoluble groups to our aims (see \cite{course} for the main results and properties of supersoluble groups). Moreover, every supersoluble brace $B$ is {\it almost polycyclic} (see~\cite{tutti23-2}), that is,~$B$ has a finite chain of ideals $$\{0\}=J_0\leq J_1\leq\ldots\leq J_n=B$$ such that, for all $0\leq i<n$, either $J_{i+1}/J_i\leq\operatorname{Soc}(B/J_i)$ and $J_{i+1}/J_i$ is finitely generated, or $J_{i+1}/J_i$ is finite. Thus, we can also take great advantage from the results in~\cite{tutti23-2}; in particular, supersoluble braces are finitely presented and satisfy the maximal condition on subbraces.

We now show that supersoluble braces encompass some relevant classes of braces (see~The\-o\-rems~\ref{fgnilp}, \ref{squarefreeorder} and~Co\-rol\-la\-ry~\ref{corsquarefreeorder}).

\begin{theo}\label{fgnilp}
Every finitely generated centrally nilpotent brace $B$ is supersoluble.
\end{theo}
\begin{proof}
Since $B$ is centrally nilpotent, it follows from \cite{Tr23}, Theorem 3.7, that $(B,+)$ is finitely generated. Now, $(B,+)$ is finitely generated and nilpotent, so it satisfies the maximal condition on subgroups, which means that every section of $(B,+)$ is finitely generated. Let $Z=\zeta(B)$. Then $Z$ is a trivial brace whose additive subgroups are ideals of $B$. Since $Z$ is additively finitely generated, its additive group is a direct sum of finitely many cyclic groups, so $Z$ has a chain of ideals whose factors have cyclic additive groups. Similarly, every section of the upper central series of $B$ can be refined to a finite chain of ideals whose factors have cyclic additive groups, and this completes the proof.
\end{proof}

\medskip

Actually, in the finite case we can precisely identify those supersoluble braces that are centrally nilpotent. As shown by the following result, they are precisely the supersoluble braces whose additive Sylow subgroups are ideals.

\begin{theo}\label{pbracecentrallynilp}
Let $p$ be a prime and let $B$ be a finite brace whose order is a power of $p$. If $B$ is supersoluble, then $B$ is centrally nilpotent. 
\end{theo}
\begin{proof}
Let $I$ be an ideal of $B$ such that $|I|=p$. Since $|B|=p^n$ for some non-negative integer $n$, it follows that $$(I,+)\leq C_{(B,+)}(B,+)\quad\textnormal{and}\quad (I,\bigcdot)\leq C_{(B,\bigcdot)}(B,\bigcdot).$$ Write $I=\langle b\rangle$ and let $G=(B,+)\rtimes_\lambda(B,\bigcdot)$; in particular, $|G|=p^{2n}$. Now, if $a\in (B,+)$, then $a\in C_G(b)$, so $a\ast b=0$. On the other hand, $$a\ast b=-a+ab-b=-a-b+ba-a+a=-a+b\ast a+a,$$ so $b\ast a=0$ and hence $b\in\operatorname{Ker}(\lambda)$. Therefore $I\leq\zeta(B)$ and $B$ is centrally nilpotent by induction on the order of $B$.
\end{proof}

\medskip

The following is an alternative interesting proof proposed by the referee.

\begin{proof}
Since $B$ is, by assumption, supersoluble and of order $p^n$ for some positive integer $n$, there exists an ideal, say $I$, of order $p$. In particular, $I$ is a minimal non-zero ideal and $B$ is left-nilpotent. We claim that $B\ast I=0$. Indeed, $B\ast I$ is a left-ideal of~$B$ contained in $I$. Hence $B\ast I=\{0\}$ or $B\ast I=I$. As $B$ is a left-nilpotent, the latter is not possible. Thus, indeed, $B\ast I=\{0\}$. So, $a\ast b=0$ for all $a\in B$ and $b\in I$. Since~$(I,\bigcdot)$ is normal in $(B,\bigcdot)$, $(I,+)$ is normal in $(B,+)$, and $(B,+)$ and $(B,\bigcdot)$ are nilpotent, we have that $I\leq Z(B)$. Thus, $b\bigcdot a=a\bigcdot b$. Then, $0=-a+a\bigcdot b-b=-a+b\bigcdot a-b=-a-b+b\bigcdot a=-a-b+b\bigcdot a-a+a=-a+b\ast a+a$. So, $b\ast a=0$ and therefore $B\ast I=\{0\}$. Hence $I\leq\zeta(B)$. The result now follows by induction on the order of $B$.
\end{proof}

\begin{theo}\label{squarefreeorder}
Let $B$ be a finite brace whose additive and multiplicative Sylow subgroups are cyclic. Then $B$ is supersoluble. 
\end{theo}
\begin{proof}
Since $(B,+)$ and $(B,\bigcdot)$ are finite groups whose Sylow subgroups are cyclic, we have that $(B,+)$ and $(B,\bigcdot)$ are supersoluble (see \cite{course}, 10.1.10). If $p$ is the largest prime dividing the order of $B$, then the set $P$ of all $p$-elements of $(B,+)$ is a characteristic subgroup of $(B,+)$ by \cite{course}, 5.4.8, which means that~$P$ is a strong left-ideal of~$B$. Therefore~$P$ is a subgroup of $(B,\bigcdot)$, and order considerations show that $P$ is the unique~Sy\-low~\hbox{$p$-sub}\-group of $(B,\bigcdot)$, so~$P$ is actually an ideal of $B$. Since~$(P,+)$ and~$(P,\bigcdot)$ are cyclic, they contain only one subgroup of order $p$; let $I$ and $J$ be the cyclic subgroup of order $p$ of $(P,+)$ and $(P,\bigcdot)$, respectively. Since $I$ is characteristic in $(P,+)$, it follows that $I=J$ is an ideal of $B$ of prime order $p$. By induction, $B/I$ is supersoluble and hence $B$ is supersoluble.
\end{proof}

\begin{cor}\label{corsquarefreeorder}
Every finite brace of square-free order is supersoluble.
\end{cor}

\medskip

Now, we prove one of the most relevant results about finite supersoluble braces, that is, the fact that they have Sylow towers. This is an immediate consequence of the following more general result. In what follows, if $n\in\mathbb{N}$, then $C_n$ denotes the cyclic group of order~$n$, and~$\mathbb{Z}$ denotes the additive group of the integers. Also, before proving the following main result, we state a result that we will frequently use in the paper.

\begin{theo}[\cite{tutti23-2}, The\-o\-rem~3.4]\label{exth3.4}
Let $B$ be an almost polycyclic brace. Then $B$ has a finite chain of ideals $$\{0\}=B_0\leq B_1\leq\ldots\leq B_n\leq B$$ such that $B_{i+1}/B_i\leq\operatorname{Soc}(B/B_i)$, $B/B_n$ is finite, and $(B_{i+1}/B_i,+)$ is torsion-free and finitely generated.
\end{theo}

\begin{theo}\label{theoincredible}
Let $B$ be a supersoluble brace. Then $B$ has a chain of ideals $$\{0\}=I_{0,0}\leq I_{0,1}\leq\ldots\leq I_{0,\ell}=I_{1,0}\leq I_{1,1}\leq\ldots\leq I_{1,m}=I_{2,0}\leq I_{2,1}\leq\ldots\leq I_{2,n}=B$$ satisfying the following properties:
\begin{itemize}
    \item $|I_{0,i+1}/I_{0,i}|$ is an odd prime number for all $0\leq i<\ell$;
    \item $|I_{0,i+1}/I_{0,i}|\geq |I_{0,i+2}/I_{0,i+1}|$ for all $0\leq i<\ell-1$;
    \item $I_{1,i+1}/I_{1,i}\leq\operatorname{Soc}(B/I_{1,i})$ and $(I_{1,i+1}/I_{1,i},+)$ is infinite cyclic for all \hbox{$0\leq i<m$;}
    \item $|I_{2,i+1}/I_{2,i}|=2$ for all $0\leq i<n$. 
\end{itemize}
\end{theo}
\begin{proof}
Given a finite chain of ideals of the type described in Definition \ref{defsuper}, we show that it is essentially possible to swap adjacent factors that are not in the order prescribed by the statement (although in one case the swap could produce some new factors of order~$2$ at the top). Then, in order to complete the proof, we only need to keep swapping the factors until there is no further possible swap. Let $H\leq K\leq L$ be ideals of $B$. 

\medskip

Suppose first $|K/H|=p$ and $|L/K|=q$ are prime numbers such that $p<q$. Since $(L/H,+)$ and $(L/H,\bigcdot)$ are supersoluble, we have (see \cite{course}, 5.4.8) that their Sylow~\hbox{$q$-sub}\-groups are normal, so each of these groups is a direct product of a cyclic group of order $p$ and a cyclic group of order $q$; in particular, both groups are nilpotent. Thus, the Sylow $q$-subgroup $Q/H$ of $(L/H,+)$ is an ideal of $B/H$. Clearly,~\hbox{$|Q/H|=q$} and $|L/Q|=p$. Hence, in this case, one can indeed swap ideals.

\medskip

Suppose now that $K/H\leq\operatorname{Soc}(B/H)$, $(K/H,+)$ is an infinite cyclic group and~\hbox{$|L/K|=p$} is an odd prime number. For the sake of clarity, we put $H=\{0\}$. Since $L/K$ is an odd prime, we have that $(L,+)$ and $(L,\bigcdot)$ are abelian groups, and furthermore $(L,+)$ and~$(L,\bigcdot)$ are either infinite cyclic or isomorphic to $C_p\times\mathbb{Z}$. 

Suppose $(L,+)$ is infinite cyclic and $(L,\bigcdot)\simeq C_p\times\mathbb{Z}$, so there exists $0\neq x\in L$ such that $x^p=0$. Choose a prime $q\neq p$. Then $L/K^{q,\bigcdot}$ is the direct product of two ideals, one of order $p$ and the other of order $q$. Clearly, $xK^{q,\bigcdot}$ is contained in the ideal of order~$p$, so $px\in K^{q,\bigcdot}$. The arbitrariness of $q$ and the fact that $\bigcap_{p\neq q\in\mathbb{P}}K^{q,\bigcdot}=\{0\}$ show that~\hbox{$px=0$.}  Because, by assumption, $(L, +)$ is infinite cyclic, this implies $x = 0$. This contradiction shows that if $(L,+)$ is infinite cyclic, then also $(L,\bigcdot)$ is such. In this case, Corollary 4.3 of \cite{stefan} yields that $L$ is abelian (as a brace). Let $(L,+)=\langle y\rangle_+$; in particular, $(K,+)=\langle y^p\rangle_+$. Now, since $(L,+)$ is infinite cyclic and $(K,+)\leq Z(B,+)$, we have that $(L,+)\leq Z(B,+)$. Moreover, if $b$ is any element of $B$, then (by Eqs. \eqref{dacitare}) $$ 0 = y^2 \ast b =y\ast (y\ast b)+2(y\ast b)=2(y\ast b)$$
as $y \ast (y \ast b) \in L \ast L = 0$ because $L$ is an ideal which is a trivial brace. More in general,  by induction one proves $$0=y^p\ast b=p(y\ast b),$$ which means that $y\ast b=0$ and hence that $L\leq\operatorname{Soc}(B)$.

We may therefore assume that $(L,+)=L_1\times L_2$, where $L_2\subseteq(\operatorname{Soc}(B),+)$ is infinite cyclic, and $L_1$ has order $p$.  In this case, we have that $L$ is a brace of abelian type such that $(L,\bigcdot)$ is not cyclic (otherwise, it would be a trivial brace by \cite{stefan}, Corollary~4.3, with elements of order $p$, a contradiction). Thus, $(L,\bigcdot)\simeq C_p\times\mathbb{Z}$, so that $L_1$ is an ideal of $B$ such that $L/L_1$ is contained in $\operatorname{Soc}(B/L_1)$ and is $1$-generator.

\medskip

Finally, suppose $|K/H|=2$, $L/K\leq\operatorname{Soc}(B/K)$ and $(L/K,+)$ is infinite cyclic. Also in this case, both $(L/H,+)$ and $(L/H,\bigcdot)$ are abelian. It follows from The\-o\-rem~\ref{exth3.4} that there is an ideal $U/H$ of $B/H$ such that $U/H\leq\operatorname{Soc}(B/H)\cap L/H$, $(U/H,+)$ is infinite cyclic, and $L/U$ is finite. Now, $L/U$ is the direct product of an even-order ideal~$P/U$ of $B/U$, and an odd-order ideal $Q/U$ of $B/U$. An application of the previous case to the chain~\hbox{$H\leq U\leq Q$} shows that $Q/H\leq\operatorname{Soc}(B/H)$ and that $(Q/H,+)$ is infinite cyclic. Since $|L/Q|$ is a power of $2$, we can refine $L/Q$ with a finite chain of ideals of order~$2$. Thus, we have swapped the types of $K/H$ and $L/K$ with some possible addition of factors of order $2$ at the top. The statement is proved.
\end{proof}

\begin{defi}
{\rm Let $B$ be a brace and $p$ a prime. We denote by $U_p^+(B)$ (resp. $U_p^{\bigcdot}(B)$) the set of all elements of $B$ whose additive (resp. multiplicative) order is finite and is divided only by primes $q>p$. In particular, $U_2^+(B)$ is the set of all odd-order elements of~$(B,+)$.}
\end{defi}

Let $B$ be a supersoluble brace, and let $p$ be a prime. Since $U_p^+(B)$ is a characteristic subgroup of $(B,+)$, it follows that $U_p^+(B)$ is a strong left-ideal of $B$. As a consequence of Theorem \ref{theoincredible}, we can say something more. Recall first that if $G$ is a group, then~$\pi(G)$ denotes the set of all prime numbers $p$ such that $G$ has an element of order~$p$.

\begin{cor}\label{psubbrace}
Let $B$ be a supersoluble brace and $p$ a prime. Then $U_p^+(B)=U_p^{\bigcdot}(B)$ is an ideal of $B$.
\end{cor}
\begin{proof}
Let $q$ be the largest odd prime in $\pi(B,+)$; of course, we may assume $p\leq q$. It follows from Theorem \ref{theoincredible} that $B$ has an ideal $I$ such that $|I|$ is a power of $q$ and both~$(B/I,+)$ and $(B/I,\bigcdot)$ have no $q$-elements. This means that $I=U_r^+(B)=U_r^{\bigcdot}(B)$ is an ideal of $B$, where~$r$ is the largest prime that is strictly less than $q$. Now, the largest prime in $\pi(B/I,+)$ is strictly less than $q$, so induction yields that $$U_p^+(B)/I=U_p^+(B/I)=U_p^{\bigcdot}(B/I)=U_p^{\bigcdot}(B)/I$$ is an ideal of $B/I$. Thus, $U_p^+(B)=U_p^{\bigcdot}(B)\trianglelefteq B$.
\end{proof}

\begin{rem}
{\rm As a consequence of Corollary \ref{psubbrace}, we have that every supersoluble brace can be embedded into the direct product of a finite supersoluble brace and a supersoluble brace $C$ with $U_2^+(C)=\{0\}$. In fact, let $B$ be our supersoluble brace, and let $I=U_2^+(B)$. By Theorem \ref{exth3.4}, $B$ has a finite-index ideal $J$ such that $(J,+)$ is torsion-free; thus, $B$ embeds in the direct product $B/I\times B/J$ by the map $b\mapsto (b+I,b+J)$. 

This remark justifies the additional hypothesis $U_2^+(B)=\{0\}$ on some of our theorems.}
\end{rem}

\medskip

Our next aim is to show that our definition of supersoluble brace is equivalent to the following (seemingly weaker) one.

\begin{defi}
{\rm A brace $B$ is said to be {\it weakly supersoluble} if there is a finite chain of ideals of~$B$ $$\{0\}=I_0\leq I_1\leq\ldots\leq I_n=B$$ such that, for all $0\leq i<n$, either $(I_{i+1}/I_i,+)$ is a central infinite cyclic subgroup of~$(B/I_i,+)$, or $I_{i+1}/I_i$ has prime order.}
\end{defi}

Clearly, every supersoluble brace is weakly supersoluble, and our next lemma is the key ingredient in showing that the converse also holds.

\begin{lem}\label{lemequivalenza}
Let $B$ be a supersoluble brace and let $I$ be an ideal of $B$.
\begin{itemize}
    \item[\textnormal{(1)}] If $I\leq Z(B,+)$ and $|I|$ is a prime, then $I\leq\operatorname{Soc}(B)$.
    \item[\textnormal{(2)}] If $I$ is abelian and $(I,+)$ is infinite cyclic, then $I\leq\operatorname{Soc}(B)$.
\end{itemize}
\end{lem}
\begin{proof}
(1)\quad It follows from Theorem \ref{exth3.4} that $B$ has an ideal $J$ such that $B/J$ is finite and $(J,+)$ is torsion-free. Now, $J\cap I=\{0\}$, so if we can prove that $I+J/J\leq\operatorname{Soc}(B/J)$, then $$[a,b]_++J=a\ast b+J=J$$ for all $a\in I$ and $b\in B$, so $[a,b]_+=a\ast b=0$ and consequently   $I\leq\operatorname{Soc}(B)$. Thus, without loss of generality we may replace~$B$ by $B/J$ and assume that $B$ is finite. If~\hbox{$p=2$,} then obviously $I\leq\zeta(B)$. Assume $p>2$, and let $q$ be the largest prime strictly less than~$p$. Put $P=U_q^+(B)$ and~\hbox{$Q=U_p^+(B)$;} in particular, $Q\leq P$ and $(P/Q,+)$ is a $p$-group. A further replacement of $B$ by $B/Q$ allows us to assume that $p$ is the biggest prime dividing the order of $B$. It follows from Theorem \ref{pbracecentrallynilp} that $I\leq\zeta(P)$. Now, if $b$ is an arbitrary element of $B$, we can write $b=c+d$, where $c\in P$ and $d$ is a $p'$-element of~$(B,+)$ of order $m$. If $I=\langle a\rangle$, then $a\in Z(B,+)$, so (by Eqs. \eqref{dacitare} and the fact that $I$ is an ideal) $$a\ast (\ell d)=\ell(a\ast d)$$ for all $\ell\in\mathbb{N}$. Thus, $$m(a\ast d)=a\ast (md)=0$$ and consequently $a\ast d=0$, since $a\ast d$ is a $p$-element of~$(B,+)$. Therefore (again by~Eqs. \eqref{dacitare} and the fact that $I$ is an ideal) $$a\ast b=a\ast(c+d)=a\ast c+a\ast d=0$$ and hence $a\in\operatorname{Soc}(B)$, as desired.

\medskip

\noindent(2)\quad It follows from Theorem \ref{theoincredible} that $|I/\operatorname{Soc}(B)\cap I|=2^n$ for some $n\in\mathbb{N}$. Since $(I,+)$ is infinite cyclic, and $\big(\operatorname{Soc}(B)\cap I,+\big)$ is central in $(B,+)$, we have that $I\leq Z(B,+)$ (recall that the only automorphism of the infinite cyclic group is the inversion). Now, let $a\in I$ and $b\in B$. Then (by Eqs. \eqref{dacitare}) $$0=\big(a^{2^n}\big)\ast b=2^n(a\ast b),$$ so $a\ast b=0$ and $a\in\operatorname{Ker}(\lambda)$. This proves that $a\in\operatorname{Soc}(B)$ and completes the proof.
\end{proof}

\medskip

Before proving the equivalence between the two definitions, we need to point out a couple of ways to exploit induction in braces $B$ whose additive group is polycyclic-by-finite; this is the case for instance of almost polycyclic braces and weakly supersoluble braces. First, one can use induction on the {\it Hirsch length}~$h$ of~$(B,+)$ (that is, the number of infinite cyclic factors appearing in a finite series of $(B,+)$ with cyclic factors). Moreover, if $n$ is the order of the largest periodic normal subgroup of $(B,+)$, then it is possible to use induction on the set of all lexicographically ordered pairs $(h,n)$ (recall that every polycyclic-by-finite group is~Hopfian). In the following, we will usually speak of ‘‘induction’’ if the context makes it clear whether we are using induction on the Hirsch length or on the pairs $(h,n)$.

\begin{theo}\label{weaksupersoluble}
Let $B$ be a brace. The following properties are equivalent:
\begin{itemize}
    \item $B$ is weakly supersoluble.
    \item $B$ is supersoluble.
\end{itemize}
\end{theo}
\begin{proof}
Suppose $B$ is weakly supersoluble, and let $I$ be an abelian ideal of $B$ such that either $I\leq Z(B,+)$ and $(I,+)$ is infinite cyclic, or $I$ has prime order. In the latter case, $B/I$ is supersoluble by induction, and so $B$ is supersoluble. Assume $(I,+)$ is infinite cyclic and $I\leq Z(B,+)$. If $p$ is any prime, then, by induction, $B/I^{p,+}$ is supersoluble, so Lemma \ref{lemequivalenza} yields that $I/I^{p,+}\leq\operatorname{Soc}(B/I^{p,+})$. Since $\bigcap_{p\in\mathbb{P}} I^{p,+}=\{0\}$, we have that $I\leq\operatorname{Soc}(B)$. The proof is complete.
\end{proof}

\medskip

Now, we deal with properties concerning the {\it index} of subbraces in a supersoluble environment. Recall first the following definition (see \cite{tutti23} and \cite{tutti23-2}).

\begin{defi}%\label{finiteindexdef}
{\rm Let $B$ be a brace. A subbrace $C$ of $B$ is said to have {\it finite index} in $B$ if both $n_+=|(B,+):(C,+)|$ and $n_{\bigcdot}=|(B,\bigcdot):(C,\bigcdot)|$ are finite; if~\hbox{$n_+=n_{\bigcdot}=n$,} we define the {\it index $|B:C|$} of $C$ in $B$ as $n$. If $C$ has not finite index, we say that $C$ has {\it infinite index}.}
\end{defi}

As shown in \cite{tutti23-2}, Corollary 3.12, if $C$ is a subbrace of an almost polycyclic brace $B$ such that $(C,+)$ (resp. $(C,\bigcdot)$) has finite index in $(B,+)$ (resp. $(B,\bigcdot)$), then $C$ has finite index in $B$, and the index of $C$ in $B$ exists --- it obviously coincides with the index of $(B,+)$ (resp.~$(B,\bigcdot)$). Since supersoluble braces are almost polycyclic, the above remark also applies to them. It has been shown in \cite{tutti23-2}, Theorem 3.16, that if $C$ is a maximal subbrace of an almost polycyclic brace $B$; then $|B:C|$ is a prime power. For supersoluble braces, we can even say that a maximal subbraces is a maximal additive (and multiplicative) subgroup (see~\cite{course}, 5.4.7, for the corresponding statement for groups).

\begin{theo}\label{maximalsubbrace}
Let $B$ be a supersoluble brace, and let $C$ be a maximal subbrace of $B$. Then $|B:C|$ is a prime number. In particular, $(C,+)$ \textnormal(resp. $(C,\bigcdot)$\textnormal) is a maximal subgroup of~$(B,+)$ \textnormal(resp. $(B,\bigcdot)$\textnormal) of prime index.
\end{theo}
\begin{proof}
It follows from \cite{tutti23-2}, Theorem 3.10, that we may assume $B$ is finite. Let $I$ be any ideal of $B$ whose order is a prime, $p$ say. Then either $I\leq C$, so an easy induction completes the proof, or $I+C=B$ and consequently $$|B:C|=|(B,+):(C,+)|=p.$$ The statement is proved.  
\end{proof}

\medskip

Note that the statement of Theorem \ref{maximalsubbrace} also holds if we replace $C$ by a maximal (strong) left-ideal (see \cite{tutti23-2}, Remark 3.17). Thus, we have the following result.

\begin{theo}
Let $B$ be a supersoluble brace. Then every maximal \textnormal(strong\textnormal) left-ideal of $B$ is a maximal subbrace of $B$ of prime index.
\end{theo}

We now show that if the index of the maximal subbrace is a small prime, then that subbrace is actually an ideal.

\begin{lem}\label{lemprimefiniteca}
Let $B$ be a finite supersoluble brace, and let $C$ be a subbrace of $B$. If $p$ is the largest prime dividing $|B:C|$, then $U_p^+(B)\leq C_B$.
\end{lem}
\begin{proof}
Let $n=|B:C|$ and let $D$ be the normal core of $(C,+)$ in $(B,+)$. Then the index of $D$ in $B$ divides $n!$. Thus, $U_p^+(B)\leq D\leq C$, and the statement is proved.
\end{proof}

\begin{theo}\label{finitecasesuperso}
Let $B$ be a finite supersoluble brace, and let $p$ be the smallest prime in $\pi(B,+)$. If $C$ is a subbrace of $B$ such that $|B:C|=p$, then $C\trianglelefteq B$. 
\end{theo}
\begin{proof}
By Lemma \ref{lemprimefiniteca}, $U_p^+(B)\leq C$, so replacing $B$ by $B/U_p^+(B)$, we may assume that the order of $B$ is a power of $p$. Then~The\-o\-rem~\ref{pbracecentrallynilp} shows that $B$ is centrally nilpotent, and an application of~\cite{tutti23},~The\-o\-rem~4.6, yields that~\hbox{$C\trianglelefteq B$.}
\end{proof}

\begin{cor}\label{corsuperindex2}
Let $B$ be a supersoluble brace, and let $C$ be a finite-index subbrace of $B$ such that $|B:C|=2$. Then $C\trianglelefteq B$.
\end{cor}
\begin{proof}
It follows from \cite{tutti23-2}, Theorem 3.10, that we may assume $B$ is finite. Then~The\-o\-rem~\ref{finitecasesuperso} yields that $C$ is an ideal of~$B$.
\end{proof}

\medskip

We end our excursion on the maximal subbraces of a supersoluble brace by noting that although the property of being supersoluble can be detected in a finite group by only looking at the index of the  maximal subgroups (see \cite{course}, 9.4.4), this is not possible for a finite brace. Note that in the examples, $C_n$ denotes the cyclic group of order $n$, while $D_n$ denotes the dihedral group of order $n$.

\begin{ex}
\label{ex:prime_index}
There exists a non-supersoluble finite brace of order $8$ whose maximal subbraces have prime index.
\end{ex}
\begin{proof}
Let 
\[ 
(B,+)  = \langle a \rangle \times \langle b \rangle \simeq \operatorname{C}_4\times \operatorname{C}_2 \quad \text{and}\quad (C,\bigcdot)\!=\! \langle x, y \mid  x^4 \!= y^2 =\! 1,\, xy = yx^3\rangle \simeq \operatorname{D}_8\] 
be groups of order $8$. The multiplicative group acts on the additive by means of the map $\lambda$ defined by
\[   
\begin{array}{lllll}
\lambda_x(a) = 3a+b,  & & \lambda_y(a) = 3a +b,   \\
\lambda_x(b) = 2a+b, & & \lambda_y(b) = b.    \\
\end{array}
\]
Consider the semidirect product $G$ of $B$ and $C$ with respect to this action (we use multiplicative notation in $G$). Then $G$ turns out to be a trifactorised group as it possesses a subgroup $D=\langle ax, a^2y\rangle$ such that $D\cap C=D\cap B=\{1\}$ and $DC=BD=G$. $D$ leads to the bijective $1$-cocycle $\delta\colon C\longrightarrow B$ with respect to $\lambda$ given by  Table~\ref{tb:prime_index}. This yields a product in $B$ and we get a brace of abelian type $(B,+,\bigcdot)$ of order $8$. This brace corresponds to \texttt{SmallBrace(8,18)} in the \textsf{Yang--Baxter} library for \textsf{GAP}.

\begin{table}[h]
\[
\begin{array}{llllllll}
g    & \delta(g) & g & \delta(g)       \\\hline
1    & 0     & y    & 2a    \\
x    & a     & xy   & 3a    \\
x^2  & b     & x^2y & 2a+b  \\
x^3  & 3a+b  & x^3y & a+b   \\\hline      
\end{array}
\]
\caption{Bijective $1$-cocycle associated with Example~\ref{ex:prime_index}}
\label{tb:prime_index}
\end{table}

From  Table~\ref{tb:prime_index}, we see that $S = \{1, x^2,y,x^2y\}$ is the unique subgroup of order $4$ in~$(C,\bigcdot)$ such that $\delta(S) \leq (B,+)$. Thus, $\delta(S)$ is a maximal subbrace of $B$. Moreover,~$\delta(xy)$ and $\delta(x^3y)$ are not elements of order $2$. Therefore, all subbraces of order $2$ of~$B$ are contained in $\delta(S)$. Hence, $\delta(S)$ is the unique maximal subbrace of $B$ and it is of prime index. Finally, we note that $B$ is not supersoluble because no subbrace of order~$2$ is~\hbox{$\lambda$-in}\-variant.  
\end{proof}

\medskip

One of the most relevant properties of almost polycyclic braces is the fact that they are {\it residually finite}, that is, the intersection of their finite-index ideals is the zero subbrace. This property allows us to approximate almost polycyclic braces by finite braces, and was the starting point of many relevant results in \cite{tutti23-2}. There are types of almost polycyclic braces that have even nicer residual properties. For example, it has been proved in \cite{tutti23-2}, Theorem 3.25, that, if $\pi$ is a set of prime numbers, then a finitely generated centrally nilpotent brace $B$ is {\it residually $\pi$-finite} (meaning that the intersection of all ideals $I$ of $B$ with $B/I$ a $\pi$-number is $\{0\}$) provided that the order of the torsion-subgroup of $(B,+)$ is a $\pi$-number. We can say something similar for supersoluble braces.

\begin{theo}\label{resfinite}
Let $B$ be a supersoluble brace, and let $p$ be the largest prime in $\pi(B,+)\cup\{2\}$. Then $B$ is residually $\pi$-finite, where $\pi$ is the set of all prime numbers $q$ with $q\leq p$.
\end{theo}
\begin{proof}
We use induction on the Hirsch length $h$ of $(B,+)$. If $h=0$, then the order of $B$ is a $\pi$-number, so we are done. Assume $h>0$ and put $U=U_2^+(B)$. It follows from Theorem \ref{theoincredible} that there is an ideal $I$ of $B$ such that $I/U\leq\operatorname{Soc}(B/U)$, $(I/U,+)$ is infinite cyclic, and $U_2^+(B/I)=\{0\}$. By Theorem \ref{exth3.4}, there is an ideal $J$ of $B$ such that $J\leq I\cap\operatorname{Soc}(B)$ and $(J,+)$ is infinite cyclic, and by Theorem \ref{theoincredible} we may even assume $U_p^+(I/J)=\{0\}$. Thus, for each $n\in\mathbb{N}$, the largest prime in $\pi\big(B/J^{2^n,\bigcdot},+\big)$ is precisely $p$, so, by induction, $B/J^{2^n,\bigcdot}$ is residually $\pi$-finite. Since $\bigcap_{n\in\mathbb{N}}J^{2^n,\bigcdot}=\{0\}$, we have that $B$ is residually $\pi$-finite.
\end{proof}

\begin{cor}
Let $B$ be a supersoluble brace such that $(B,+)$ is torsion-free. Then the intersection of all ideals $I$ of $B$ with $|B/I|$ a power of $2$ is $\{0\}$.
\end{cor}

Some of the most interesting results for supersoluble braces are connected with nilpotency properties. For example, an obvious application of Lemma \ref{lemequivalenza} allows us to identify the supersoluble braces with a finite multipermutational level by only looking at their additive groups.

\begin{theo}\label{rnilsupersoluble}
Let $B$ be a supersoluble brace. Then $(B,+)$ is nilpotent if and only if there is~\hbox{$n\in\mathbb{N}$} such that $B=\operatorname{Soc}_n(B)$.
\end{theo}

Conversely, if we know our brace has finite multipermutational level, then we can detect supersolubility by only looking at the multiplicative group. This is the content of our next result.

\begin{theo}\label{supersolublegroup}
Let $B$ be a brace such that $B=\operatorname{Soc}_n(B)$ for some~\hbox{$n\in\mathbb{N}$.} If $(B,\bigcdot)$ is supersoluble, then $B$ is supersoluble.
\end{theo}
\begin{proof}
Clearly, both the additive and multiplicative groups are polycyclic-by-finite. Let $S=\big(\operatorname{Soc}(B),\bigcdot\big)$. Since $(B,\bigcdot)$ is supersoluble, we can find a non-trivial cyclic normal subgroup $X$ of $(B,\bigcdot)$ in $S$. By Lemma 1.10 of \cite{Cedo}, we have that $X$ is an ideal of $B$, so an easy induction on the Hirsch length of~$(B,+)$ yields that $B/X$ is supersoluble. Therefore~$B$ is supersoluble and the statement is proved.
\end{proof}

\medskip

The above two results show that in the special case of a supersoluble brace of nilpotent type, the structure of the multiplicative group plays a major role. Under these circumstances, it is possible to extend many relevant results of supersoluble groups to supersoluble braces. We now give the reader a taste of this fact by showing how the following two results of Reinhold Baer and Bertram Wehrfritz, respectively, are immediately extended to braces:
\begin{itemize}
    \item Let $G$ be a group with a nilpotent derived subgroup. If $G$ is the product of two normal supersoluble subgroups, then $G$ is supersoluble (see \cite{baersuper} and \cite{infiniteminimal}).
    \item Let $G$ be a finitely generated group. If every $2$-generator subgroup of $G$ is supersoluble, then $G$ is  supersoluble (see \cite{Weh}, Lemma 3.5).
\end{itemize}

\begin{theo}\label{baersuperso}
Let $B$ be a brace with $B=\operatorname{Soc}_n(B)$ for some $n\in\mathbb{N}$.
\begin{itemize}
    \item[\textnormal{(1)}] Let $I$ and $J$ be supersoluble ideals of $B$ such that $B=I+J$. If $[I,J]_{\bigcdot}$ is nilpotent, then $B$ is supersoluble.
    \item[\textnormal{(2)}] Suppose $B$ is finitely generated. If every $2$-generator subbrace of $B$ is supersoluble, then~$B$ is supersoluble.
\end{itemize}
\end{theo}
\begin{proof}
For the sake of clarity we prove (2). Since every $2$-generator multiplicative subgroup is contained in a $2$-generator subbrace, it follows from the above mentioned result that $(B,\bigcdot)$ is locally supersoluble. We claim by induction on $n$ that $B$ is supersoluble. If $n\leq 1$, there is nothing to prove. Suppose $n>1$. By induction, $B/\operatorname{Soc}(B)$ is supersoluble, so it is finitely presented. Therefore, Corollary 3.3 of \cite{Tr23} yields that $S=\operatorname{Soc}(B)$ is finitely generated as an ideal of~$B$ by a set $E$. Let $F$ be a finite subset of~$B$ that generates both $(B,\bigcdot)$ and $(B,+)$ modulo $(S,\bigcdot)$ and $(S,+)$, respectively.

Now, if $X$ is the multiplicative subgroup generated by $E$ and $F$, then $X$ is supersoluble, so every subgroup of $X$ is finitely generated. Let $L$ be the smallest normal subgroup $L$ of $X$ containing $E$. Then $(L,\bigcdot)$ is normal in $(B,\bigcdot)$, so $L$ is an ideal of $B$, being contained in $S$. Since $S$ is the smallest ideal of $B$ containing $E$, we have that $L=S$. This means that $(S,\bigcdot)$ is finitely generated, so also $(B,\bigcdot)$ is finitely generated. Therefore $(B,\bigcdot)$ is supersoluble and $B$ is supersoluble as well by Theorem \ref{supersolublegroup}.
\end{proof}

\medskip

The following example show that the hypothesis the brace has finite multipermutational level cannot be omitted in  Theorem \ref{baersuperso} (1).

\begin{ex}\label{ex:fact_superS}
There exists a finite brace $B$ of order $32$ containing two supersoluble ideals $I$ and $J$ such that:
\begin{itemize}
    \item $B=I+J$.
    \item $\partial(B):=\langle B\ast B,[B, B]_+\rangle_+$ is centrally nilpotent.
    \item $B$ is not supersoluble.
\end{itemize}
\end{ex}
\begin{proof}
Let 
\[\begin{array}{l}
(B,+)  = \langle a \rangle \times \langle b \rangle \times \langle c \rangle \times \langle d \rangle \times \langle e \rangle \simeq \operatorname{C}_2^5 \quad \text{and}\\[0.2cm]
(C,\bigcdot)\!=\! \langle x, y, z \mid x^4 \!= y^4 \!= z^2 = 1,\, xy = yx,\, zxz = x^3y^2,\,  zyz = x^2y\rangle\!\simeq\! \big(\operatorname{C}_4 \times \operatorname{C}_4\big)\!\rtimes\!\operatorname{C}_2.
\end{array}
\]
Of course, $(B,+)$ and $(C,\bigcdot)$ are groups of order~$32$. The multiplicative group acts on the additive by means of the map $\lambda$ defined as
\[   
\begin{array}{lllll}
\lambda_x(a) = c+d+e  & & \lambda_y(a) = a     & & \lambda_z(a) = a\\
\lambda_x(b) = c+e  & & \lambda_y(b) = a+c+e      & & \lambda_z(b) = b+c+d+e\\
\lambda_x(c) = a+c & & \lambda_y(c) = a+b+c+d & & \lambda_z(c) = a+b+c+d\\
\lambda_x(d) = a+b  & & \lambda_y(d) = b     & & \lambda_z(d) = c+e \\
\lambda_x(e) = a+b+c & & \lambda_y(e) = a+e  & & \lambda_z(e) = a+b+e\\
\end{array}
\]
Consider the semidirect product $G$ of $B$ and $C$ with respect to this action (we use multiplicative notation in $G$). Then $G$ turns out to be a trifactorised group as it possesses a subgroup $D=\langle acx, acey, az\rangle$ such that $D\cap C=D\cap B=\{1\}$, $DC=BD=G$. Thus, there is a bijective $1$-cocycle $\delta\colon C\longrightarrow B$ with respect to $\lambda$ given by  Table~\ref{tb:fact_superS}. This yields a product in $B$ and we get a brace of abelian type $(B,+,\bigcdot)$ of order $32$. This brace corresponds to \texttt{SmallBrace(32, 25055)} in the \textsf{Yang--Baxter} library for \textsf{GAP}.
\begin{table}[h]
\[
\begin{array}{llllllll}
g     & \delta(g) & g        & \delta(g) & g   & \delta(g)& g         & \delta(g) \\\hline
1     & 0        & y^2     & b+d      & z      & a        & y^2z     & a+b+d    \\
x     & a+c      & xy^2    & b+e      & xz     & a+d+e    & xy^2z    & b+c+d    \\
x^2  & c+d+e    & x^2y^2 & b+c+e    & x^2z  & a+c+d+e  & x^2y^2z & a+b+c+e  \\
x^3  & c        & x^3y^2 & a+b+e    & x^3z  & d+e      & x^3y^2z & a+b+c+d  \\
y     & a+c+e    & y^3     & b        & yz     & c+e      & y^3z     & a+b      \\
xy    & a+b+d+e  & xy^3    & a+e      & xyz    & a+b+c    & xy^3z    & a+c+d    \\
x^2y & a+d      & x^2y^3 & b+c+d+e  & x^2yz & d        & x^2y^3z & a+b+c+d+e    \\
x^3y & b+d+e    & x^3y^3 & e        & x^3yz & b+c      & x^3y^3z & c+d\\\hline       
\end{array}
\]
\caption{Bijective $1$-cocycle associated with Example~\ref{ex:fact_superS}}
\label{tb:fact_superS}
\end{table}

Ideals of $B$ of order $16$ are given by the $\lambda$-invariant subgroups $I = \langle a,c,b+d,b+e\rangle_+$, $J = \langle a,b,d,c+e\rangle_+$ and $K = \langle a, b+c, b+d, e\rangle_+$. In addition, $B$ has only one more proper ideal, $L = \langle a, b+d, b+c+e\rangle_+$ of order $8$. Thus, $B$ is not supersoluble.

Note also that $L$, $L_2 = \langle a+b+d, b+c+e\rangle_+$ and $L_3 = \langle b+c+e\rangle_+$ are ideals of both $I$ and $J$, so $I$ and $J$ are supersoluble braces.

Since $B$ is of abelian type, $\partial(B) = B \ast B = L$. Therefore, we also conclude that $\partial(B)$ is centrally nilpotent as it has a central series of ideals (see also Theorem \ref{pbracecentrallynilp}).
\end{proof}

\medskip

Now, we turn to left-nilpotency. As mentioned in the introduction, it has been proved in \cite{Cedo}, Theorem 4.8, that a finite brace of nilpotent type is left-nilpotent if and only if its multiplicative group is nilpotent. It has been proved in \cite{tutti23-2}, Theorem 3.30, that we can replace ‘‘finite brace’’ by ‘‘almost polycyclic brace’’ in the above result, but unfortunately we cannot replace ‘‘finite brace of nilpotent type’’ by ‘‘supersoluble brace’’. In fact, the finite brace of order $6$ corresponding to~\texttt{SmallSkewbrace}(6,3) in the~\texttt{Yang--Baxter} library \cite{VendraminKonovalov22-YangBaxter-0.10.2} for \texttt{GAP} \cite{GAP4-12-2} is an example of non-left-nilpotent supersoluble brace with a nilpotent multiplicative group. Nevertheless, it turns out that one of the implications holds with a little addition.

\begin{lem}\label{3.33}
Let $B$ be a left-nilpotent brace.

\begin{itemize}
    \item[\textnormal{(1)}] $\operatorname{Soc}(B)\leq\overline{\zeta}(B)$. Moreover, if $\operatorname{Soc}(B)$ is finitely generated, then $\operatorname{Soc}(B)\leq\zeta_m(B)$ for some $m\in\mathbb{N}$.
    \item[\textnormal{(2)}] If $I=\langle b\rangle$ is an ideal of $B$ of prime order, then $\lambda_a(b)=b$ for all $a\in B$. 
\end{itemize}
\end{lem}
\begin{proof}
(1)\quad We claim that $\operatorname{Soc}(B)/I\cap \zeta(B/I)\neq\{0\}$ for every ideal $I$ of $B$ such that $I<\operatorname{Soc}(B)$. For the sake of clarity, we prove the previous claim when $I=\{0\}$. Thus, let $0\neq b\in \operatorname{Soc}(B)$. Since $B$ is left-nilpotent, there are elements $b_1,\ldots,b_\ell$ of $B$ such that $$c=b_1\ast(b_2\ast(\ldots \ast (b_\ell\ast b)))\ldots)\neq0$$ and $a\ast c=0$ for all $a\in B$. It follows from Lemma 1.10 of \cite{Cedo} that $c\in\zeta(B)$, and the claim is proved.

The previous claim enables us to construct an ascending chain of ideals of $B$ whose factors are central in $B$. This certainly implies that $\operatorname{Soc}(B)\leq\overline{\zeta}(B)$. 

Finally, if $\operatorname{Soc}(B)$ is finitely generated, then it satisfies the maximal condition on subbraces, so the previous ascending series stops after finitely many steps and $\operatorname{Soc}(B)$ is contained in some finite term of its upper central series.

\medskip

\noindent(2)\quad If $a\ast b\neq 0$, then $I=\langle a\ast b\rangle$, so $a\ast b=nb$ for some $n$ which is prime to $p$. It follows that $a\ast (a\ast b)=n(a\ast b)\neq 0$ and hence that $I=\langle a\ast (a\ast b)\rangle$. Continuing in this way, we contradict the left-nilpotency of $B$. Thus $\lambda_a(b)=b$ for all $a\in B$. 
\end{proof}

\begin{theo}\label{leftnilptheo}
Let $B$ be a supersoluble brace. If $B$ is left-nilpotent and $\operatorname{Ker}(\lambda)\leq Z_n(B,\bigcdot)$ for some $n\in\mathbb{N}$, then $(B,\bigcdot)$ is nilpotent.
\end{theo}
\begin{proof}
Let $$\{0\}=I_0\leq I_1\leq\ldots\leq I_\ell=B$$ be a finite chain of ideals of $B$ such that, for each $0\leq i<\ell$, either $I_{i+1}/I_i\leq\operatorname{Soc}(B/I_i)$ and $\big(I_{i+1}/I_i,\bigcdot\big)$ is infinite cyclic, or $I_{i+1}/I_i$ has prime order. By Lemma \ref{3.33},  $(B,\bigcdot)$ stabilizes the finite series $$\{0\}=\big(I_0,+\big)\leq \big(I_1,+\big)\leq\ldots\leq \big(I_\ell,+\big)=(B,+)$$ in the natural semidirect product $(B,+)\rtimes_\lambda(B,\bigcdot)$. Thus,~$(B,\bigcdot)/\operatorname{Ker}(\lambda)$ is nilpotent by a well-known theorem of Philip Hall (see \cite{hall}, Theorem 1, and \cite{newell}, Theorem), and consequently~$(B,\bigcdot)$ is nilpotent.
\end{proof}

\begin{cor}
Let $B$ be a supersoluble brace with $\operatorname{Ker}(\lambda)=\{0\}$. If $B$ is left-nilpotent, then~$(B,\bigcdot)$ is nilpotent.
\end{cor}

\begin{cor}\label{corchepoidicoserve}
Let $B$ be a supersoluble brace with $\operatorname{Ker}(\lambda)\leq\operatorname{Soc}_n(B)$ for some $n\in\mathbb{N}$. If~$B$ is left-nilpotent, then $(B,\bigcdot)$ is nilpotent.
\end{cor}
\begin{proof}
By Lemma \ref{3.33}, $\operatorname{Ker}(\lambda)$ is contained in some finite term of the upper central series of $(B,\bigcdot)$, so we may apply Theorem \ref{leftnilptheo}.
\end{proof}

\medskip

The last nilpotency concept we consider is central nilpotency. In this context, we first note that a further consequence of the previous results connects left-nilpotency and central nilpotency.

\begin{cor}\label{3.35}
Let $B$ be a supersoluble brace of nilpotent type. Then $B$ is left-nilpotent if and only if $B$ is centrally nilpotent.
\end{cor}
\begin{proof}
By \cite{periodici}, Corollary 2.15, a brace of nilpotent type is centrally nilpotent if and only if it is left and right nilpotent. Then, the result follows from Corollary \ref{corchepoidicoserve} and Theorem \ref{rnilsupersoluble}.
\end{proof}

\medskip

However, one of the most interesting results concerns with the Fitting ideal.

\begin{theo}\label{theobcentrally}
Let $B$ be a supersoluble brace, and let $I$ be a centrally nilpotent ideal of $B$. Then $I$ is $B$-centrally nilpotent.
\end{theo}
\begin{proof}
Let $$\{0\}=J_0\leq J_1\leq\ldots\leq J_n=B$$ be a finite chain of ideals of $B$ such that, for each $0\leq i<n$, either $J_{i+1}/J_i\leq\operatorname{Soc}(B/I_i)$ and $(J_{i+1}/J_i,+)$ is an infinite cyclic group, or $J_{i+1}/J_i$ has prime order. Consider the induced chain $$\{0\}=J_0\cap I=:H_0\leq H_1:=J_1\cap I\leq\ldots\leq H_n:=J_n\cap I=I$$ of ideals of $B$. In order to show that $I$ is $B$-centrally nilpotent, we only need to prove that
$$H_{n-i+1}/H_{n-i}\leq\zeta(I/H_{n-i})$$ for all $1\leq i\leq n$.  Now, $H_{n-i+1}/H_{n-i}$ has either prime order, or is contained in the socle of $I/H_{n-i}$ and its additive group is infinite cyclic. In the former case, the quotient $H_{n-i+1}/H_{n-i}$ is a minimal ideal of the centrally nilpotent brace~$I/H_{n-i}$, so~$H_{n-i+1}/H_{n-i}\leq\zeta(I/H_{n-i})$ (see \cite{tutti23}, Theorem 4.6). Suppose the latter condition holds. Since $(I/H_{n-i},\bigcdot)$ is nilpotent, and $(H_{n-i+1}/H_{n-i},\bigcdot)$ is an infinite cyclic normal subgroup, we have that $$(H_{n-i+1}/H_{n-i},\bigcdot)\leq Z(I/H_{n-i},\bigcdot).$$ Thus again $H_{n-i+1}/H_{n-i}\leq\zeta(I/H_{n-i})$ and the theorem is proved.
\end{proof}

\begin{cor}\label{fittingideal}
Let $B$ be a supersoluble brace. Then the Fitting ideal of $B$ is centrally nilpotent and coincides with the sum of all ideals of $B$ that are centrally nilpotent.
\end{cor}
\begin{proof}
This follows from Theorem \ref{theobcentrally} and the maximal condition on subbraces.~\end{proof}

\begin{rem}
{\rm It follows from Theorem \ref{rnilsupersoluble} and the usual Fitting theorem for groups that every supersoluble brace contains a largest ideal that has finite multipermutational level.}
\end{rem}

Although the Fitting ideal of a supersoluble brace behaves well under many respects, there are a couple of issues. The first concerns with the Fitting ideal of an ideal.

\begin{ex}
\label{ex:fitting}
There exists a finite supersoluble brace $B$ of order $24$ such that: 
\begin{itemize}
    \item $B=\operatorname{Soc}_n(B)$ for some $n\in\mathbb{N}$.
    \item There is an ideal $I$ of $B$ such that $\operatorname{Fit}(I)$ is not an ideal of $B$.
\end{itemize}
\end{ex}
\begin{proof}
Let 
\[\begin{array}{l}
(B,+)  = \langle a \rangle \times \langle b \rangle  \simeq \operatorname{C}_{12}\times \operatorname{C}_2 \quad \text{and}\\[0.2cm]
(C,\bigcdot)\!=\! \big(\langle x\rangle\rtimes\langle y \rangle\big) \times \langle z\rangle \times 	\langle t\rangle \simeq \operatorname{Sym}(3) \times \operatorname{C}_2\times \operatorname{C}_2
\end{array}
\]
Of course, $(B,+)$ and $(C,\bigcdot)$ are groups of order~$24$. The multiplicative group acts on the additive by means of the map $\lambda$ defined by
\[   
\begin{array}{llllll l}
\lambda_x(a) = a,  & & \lambda_y(a) = 5a,     & & \lambda_z(a) = 7a, &  & \lambda_t(a) = 7a, \\
\lambda_x(b) = b,  & & \lambda_y(b) = b,     & & \lambda_z(b) = b, & & \lambda_t(b) = 6a + b.\\
\end{array}
\]
Consider the semidirect product $G$ of $B$ and $C$ with respect to this action (we use multiplicative notation in $G$). Then $G$ turns out to be a trifactorised group as it possesses a subgroup $D=\langle a^8x, a^6y, bz, a^3t\rangle$ such that $D\cap C=D\cap B=\{1\}$, $DC=BD=G$. Thus, there is a bijective $1$-cocycle $\delta\colon C\longrightarrow B$ with respect to $\lambda$ given by  Table~\ref{tb:fitting}. This yields a product in $B$ and we get a brace of abelian type $(B,+,\bigcdot)$ and of order $24$. This brace corresponds to \texttt{SmallBrace(24, 629)} in the \textsf{Yang--Baxter} library for \textsf{GAP}.
\begin{table}[h]\label{tb:fitting}
\[
\begin{array}{llllllll}
g     & \delta(g) & g    & \delta(g) & g   & \delta(g) & g         & \delta(g) \\\hline
1     & 0    & z     & b    & t     & 3a   & zt      & 9a+b    \\
x     & 8a   & xz    & 8a+b & xt    & 11a  & xzt     & 5a+b   \\
x^2  & 4a    & x^2z  & 4a+b & x^2t  & 7a   & x^2zt   & a+b  \\
y    & 6a    & yz    & 6a+b & yt    & 9a   & yzt     & 3a+b  \\
xy    & 2a   & xyz   & 2a+b & xyt   & 5a   & xyzt    & 11a+b      \\
x^2y  & 10a  & x^2yz & 10+b & x^2yt & a    & x^2yzt  & 7a+b    \\\hline       
\end{array}
\]
\caption{Bijective $1$-cocycle associated with Example~\ref{ex:fitting}}
\end{table}

From the $\lambda$-action described, it follows that $\operatorname{Soc}(B) = \langle 4a\rangle_+$, $\operatorname{Soc}_2(B) = \langle 2a \rangle_+$ and $\operatorname{Soc}_3(B) = B$.

Now, take $I:= \langle 2a, b\rangle_+$, which is an ideal of $B$. Observe that the multiplicative group of $I$ is $\delta^{-1}(I) = \langle x, y, z\rangle \simeq \operatorname{D}_{12}$, which is not nilpotent. Thus, $I$ is not centrally nilpotent. Moreover, $\langle 4a, b\rangle_+$ is an ideal of $I$ which is centrally nilpotent as it is a trivial brace of order $6$. Therefore, $\operatorname{Fit}(I) = \langle 4a, b \rangle_+$. But, $\operatorname{Fit}(I)$ is not an ideal of $B$, as it is not $\lambda$-invariant.\end{proof}

\medskip

It is well-known that the Fitting subgroup of a supersoluble group contains the derived subgroup (see, for example, \cite{course}, 5.4.10). Our next result shows that the {\it derived ideal} \hbox{$\partial(B)=\langle B\ast B,[B,B]_+\rangle_+$} of a supersoluble brace $B$ has always finite multipermutational level.

\begin{theo}\label{derivedisrightnilp}
Let $B$ be a supersoluble brace. Then $\partial(B)=\operatorname{Soc}_n\big(\partial(B)\big)$ for some $n\in\mathbb{N}$.
\end{theo}
\begin{proof}
It follows from Lemma \ref{supersolublebraceisgroup} that the semidirect product \hbox{$G=(B,+)\rtimes_\lambda(B,\bigcdot)$} is supersoluble. Thus,~$[G,G]$ is nilpotent. On the other hand, $[G,G]\cap (B,+)$ contains~$[B,B]_+$ and $B\ast B$, so the additive group of $\partial(B)$ is nilpotent. By~The\-o\-rem~\ref{rnilsupersoluble},~$\partial(B)$ has finite multipermutational level, and the statement is proved.
\end{proof}

\medskip

On the other hand, the following example shows that we cannot hope for central nilpotency in Theorem \ref{derivedisrightnilp}.

\begin{ex}\label{badbehderideal}
There exists a finite supersoluble brace $B$ of order $12$ such that: 
\begin{itemize}
    \item $B=\operatorname{Soc}_n(B)$ for some $n\in\mathbb{N}$.
    \item $(B\ast B,\bigcdot)$ is not nilpotent; in particular, $\partial(B)$ is not centrally nilpotent.
\end{itemize}
\end{ex}
\begin{proof}
Let $(B,+) = \langle a \rangle \simeq \operatorname{C}_{12}$ and $(C,\bigcdot) = \langle \sigma, \tau\mid \sigma^6 = \tau^2 = 1,\, \sigma \tau = \tau\sigma^{-1}\rangle \simeq \operatorname{D}_{12}$ be groups of order $12$. There exists an action $c \in (C,\bigcdot) \mapsto \lambda_c \in \operatorname{Aut}(B,+)$ given by
\[ \lambda_{\sigma}(a) = 7a\quad\textnormal{and}\quad \lambda_{\tau}(a) = 5a.\]
If we consider the semidirect product $G$ of $B$ and $C$ with respect to this action, then~$G$ turns out to be a trifactorised group, as it possesses a subgroup $D=\langle a^5\sigma, a^6\tau\rangle$ such that $D\cap C=D\cap B=\{1\}$, $DC=BD=G$. Thus, there is a bijective $1$-cocycle $\delta\colon C\longrightarrow B$ with respect to $\lambda$ given by  Table~\ref{tb:badbehderideal}. This yields a product in $B$ and we get a brace of abelian type $(B,+,\bigcdot)$ and of order $12$, corresponding to \texttt{SmallSkewbrace(12, 15)} in the~\textsf{Yang--Baxter} library for \textsf{GAP}.
\begin{table}[h!]
\[
\begin{array}{llll}
c   & \delta(c) & c   & \delta(c) \\\hline 
1        & 0   & \tau         & 6a    \\
\sigma   & 5a  & \sigma\tau   & 11a   \\
\sigma^2 & 4a  & \sigma^2\tau & 10a  \\
\sigma^3 & 9a  & \sigma^3\tau & 3a  \\
\sigma^4 & 8a  & \sigma^4\tau & 2a  \\
\sigma^5 & a   & \sigma^5\tau & 7a \\\hline        
\end{array}\]
\caption{Bijective $1$-cocycle associated with Example~\ref{badbehderideal}}
\label{tb:badbehderideal}
\end{table}

Since $\operatorname{Ker} \lambda = \langle \sigma^2\rangle$, it holds that $\langle 4a\rangle_+ = \operatorname{Soc}(B)$. Then, $\overline{B}:= B/\operatorname{Soc}(B)$ is a brace of order $4$. Moreover, if we consider $\overline{\lambda}\colon \big(C/\langle\sigma^2\rangle,\bigcdot\big) \rightarrow \operatorname{Aut}(B/\langle 4a \rangle, +)$, it follows that $\operatorname{Ker}\overline{\lambda} = \big\langle \tau + \langle \sigma^2\rangle \big\rangle$. Therefore, $\operatorname{Soc}_2(B) = \langle 2a \rangle_+$ and $B/\operatorname{Soc}_2(B)$ is a (trivial) brace of order $2$. Hence, $\operatorname{Soc}_3(B) = B$.

On the other hand, $B \ast B = \langle 2a \rangle_+$ and $\delta^{-1}\big(\langle 2a \rangle_+\big) = \langle \sigma^2, \tau \rangle \simeq \operatorname{Sym}(3)$, which is not nilpotent. Hence, $\partial(B) = B \ast B$ is not centrally nilpotent.
\end{proof}

\medskip

Despite the above example, it turns out that the~Fitting ideal of a supersoluble brace has always finite index (see \cite{course}, 5.4.10, for the analogue property for supersoluble groups).

\begin{theo}\label{fittingfiniteindex}
Let $B$ be a supersoluble brace. Then $B/\operatorname{Fit}(B)$ is finite.
\end{theo}
\begin{proof}
By hypothesis, $B$ has a chain of ideals $$\{0\}=I_0\leq I_1\leq\ldots\leq I_n=B$$ such that, for each $0\leq i<n$, either $I_{i+1}/I_i$ is contained in $\operatorname{Soc}(B/I_i)$ and is~\hbox{$1$-ge}\-ne\-ra\-tor, or $I_{i+1}/I_i$ has prime order. In particular, for each $0\leq i<n$, we have that~\hbox{$C_{(B,\bigcdot)}(I_{i+1}/I_i,\bigcdot)$} and $C_{(B,+)}(I_{i+1}/I_i,+)$ have finite index in $(B,\bigcdot)$ and $(B,+)$, respectively. By Theorem 3.10 of \cite{tutti23-2}, there is a finite-index ideal $J$ of $B$ that is contained in all these centralizers. Intersecting the chain of the $I_i$'s with $J$, we see that~$J$ is centrally nilpotent (one also needs to use Lemma \ref{lemequivalenza}). Thus, $J\leq\operatorname{Fit}(B)$ by~Co\-rol\-la\-ry \ref{fittingideal} and the statement is proved.
\end{proof}

\medskip

On some occasions we can say something more about the index of the Fitting ideal.

\begin{lem}\label{torsionfreesupersocatena}
Let $B$ be a supersoluble brace, and let $I$ be a centrally nilpotent ideal of $B$ such that $(I,+)$ is torsion-free. Then there is a finite chain $$\{0\}=I_0\leq I_1\leq\ldots\leq I_n=I$$ of ideals of $B$ such that $I_{i+1}/I_i\leq\operatorname{Soc}(B/I_i)\cap\zeta(I/I_i)$ and $(I_{i+1}/I_i,+)$ is infinite cyclic.
\end{lem}
\begin{proof}
By Theorem \ref{theoincredible}, there is an ideal $J$ of $B$ contained in $I$ such that $J\leq\operatorname{Soc}(B)$ and~$(J,+)$ is infinite cyclic. Since~\hbox{$(I,\bigcdot)$} is nilpotent and~\hbox{$(J,\bigcdot)$} is a normal infinite cyclic subgroup of $(B,\bigcdot)$, we have that $J\leq Z(I,\bigcdot)$, so $J\leq\zeta(I)$.

In order to use induction (consequently completing the proof), we should find such an ideal $J$ with an additional property: $(I/J,+)$ must be torsion-free. Since $I$ is centrally nilpotent, the set $T/J$ of all periodic elements of $(I/J,+)$ is an ideal of $I/J$ and coincides with the set of all periodic elements of $(I/J,\bigcdot)$ (see \cite{tutti23}, Theorem 4.12). Thus, $T/J$ is also an ideal of~$B/J$. Now, since $(T,+)$ is torsion-free, central-by-finite and has Hirsch length $1$, we have that~$(T,+)$ is infinite cyclic. Moreover, $T\leq\zeta(I)$ by \cite{tutti23}, Co\-rol\-la\-ry~4.17, so, in particular,~$T$ is abelian. Also, since $(T,+)$ is infinite cyclic and normal in~$(B,+)$, we have $T\leq Z(B,+)$. Let~\hbox{$x\in T$} and $b\in B$. Then $$m(x\ast b)=(x^m)\ast b=0,$$ where $m=|T/J|$, and hence $x\ast b=0$. Therefore $x\in\operatorname{Soc}(B)$ and the statement is proved.
\end{proof}

\begin{lem}\label{extensionfitting}
Let $B$ be a supersoluble brace. If $I$ is any centrally nilpotent ideal of $B$ such that $U_2^+(I)=\{0\}$, then $I\leq\zeta_n(U)$ for some $n\in\mathbb{N}$, where $U/I=U_2^+(B/I)$. 
\end{lem}
\begin{proof}
Since every ideal of order $2$ of a brace is always contained in the centre, it follows from Theorem \ref{theoincredible} that we may assume $(I,+)$ is torsion-free. Now, by~Lem\-ma~\ref{torsionfreesupersocatena}, there is a chain $$\{0\}=I_0\leq I_1\leq\ldots\leq I_n=I$$ of ideals of $B$ such that $I_{i+1}/I_i\leq\operatorname{Soc}(B/I_i)\cap\zeta(I/I_i)$ and $(I_{i+1}/I_i,+)$ is infinite cyclic. By induction on~$n$, we only need to show that $I_1\leq\zeta(U)$. 

Now, let $x\in U$. Then there are $y\in U$ and $z\in I$ such that $x=y^2z$. Since $(I_1,\bigcdot)$ is infinite cyclic, we have $[y^2,I_1]_{\bigcdot}=\{0\}$, so also $[x,I_1]_{\bigcdot}=\{0\}$. This implies $I_1\leq Z(U,\bigcdot)$ and consequently $I_1\leq\zeta(U)$.
\end{proof}

\begin{theo}\label{fittingpowerof2}
Let $B$ be a supersoluble brace such that $U_2^+(B)=\{0\}$. Then $|B/\operatorname{Fit}(B)|$ is a power of $2$.
\end{theo}
\begin{proof}
Let $F=\operatorname{Fit}(B)$. It follows from Corollary~\ref{fittingideal} and The\-o\-rem~\ref{fittingfiniteindex} that $F$ is a centrally nilpotent ideal of $B$ and $B/F$ is finite. Moreover,~Lem\-ma~\ref{extensionfitting} yields that $U_2^+(B/F)=\{0\}$, so $|B/F|$ is a power of $2$.
\end{proof}

\begin{cor}
Let $B$ be a supersoluble brace such that $(B,+)$ is torsion-free. Then $|B/\operatorname{Fit}(B)|$ is a power of $2$.
\end{cor}

\medskip

So far, we have shown that the Fitting ideal has always finite index in a supersoluble brace, and that the index can sometimes be a power of $2$. Our last main result in the context of central nilpotency deals with the special case in which the supersoluble brace is finite over some term of its upper central series. Recall first that a well-known theorem of Issai Schur states that the commutator subgroup of a group is finite provided that the group is finite over the centre. This theorem has been generalized to braces in \cite{Schur}, Theorem 5.4. Schur's theorem admits a very relevant generalization: in fact, it has been proved by Reinhold Baer that the $(i+1)$th term of the lower central series of a group is finite provided that the group is finite over the $i$th term of the upper central series (see \cite{Rob72}, Corollary 2 to Theorem 4.21). It is not clear if such a more general result holds for arbitrary braces, but in some special circumstances this is true (see for example \cite{colazzo}, Theorem 3.37). Here, we provide a further example of a class of braces in which the theorem of Baer holds.

\begin{theo}\label{schurtheor}
Let $B$ be a supersoluble brace such that $B/\zeta_n(B)$ is finite for some $n\in\mathbb{N}$. Then $\Gamma_{n+1}(B)$ is finite.
\end{theo}
\begin{proof}
Let $T$ be the largest finite ideal of $B$ (it exists because $B$ satisfies the maximal condition on subbraces). Replacing $B$ by $B/T$, we may assume that $T=\{0\}$. In particular, the additive group of $\zeta(B)$ is torsion-free, and consequently the additive group of every factor of the upper central series of $B$ is torsion-free (see~\cite{tutti23},~The\-o\-rem~4.16). Then, also using Theorem \ref{theoincredible}, we see that~$B$ has a finite chain of ideals $$\{0\}=I_0\leq I_1\leq \ldots\leq I_\ell=J_0\leq J_1\leq\ldots\leq J_m=U\leq B,$$ where 
\begin{itemize}
    \item[(a)] $I_{i+1}/I_i\leq\zeta(B/I_i)$ and $(I_{i+1}/I_i,+)$ is infinite cyclic for all $0\leq i<\ell$,
    \item[(b)] $|J_{i+1}/J_i|$ is an odd prime for all $0\leq i<m$, and
    \item[(c)] $|B/U|$ is a power of $2$.
\end{itemize}

We now claim that there is a finite chain of ideals of $B$ contained in $J_1$ whose factors satisfy (a). Write  $p=|J_1/J_0|$. It follows from Theorem \ref{theoincredible} that only one of the following two possibilities holds:
\begin{itemize}
    \item[(1)] $J_1/I_{\ell-1}\leq\operatorname{Soc}(B/I_{\ell-1})$ and $\big(J_1/I_{\ell-1},+\big)$ is infinite cyclic.
    \item[(2)] $J_1/I_{\ell-1}=I_\ell/I_{\ell-1}\times P/I_{\ell-1}$, where $P/I_{\ell-1}$ is an ideal of $B/I_{\ell-1}$ having order $p$.
\end{itemize}

In case (2), the claim follows by induction on $\ell$ (note that the base case $\ell=0$ implies that $B=\{0\}$, so we may assume $\ell\geq1$). In case (1), $(J_1/I_{\ell-1},\bigcdot)$ is infinite cyclic, and $(I_\ell/I_{\ell-1},\bigcdot)\leq Z(B/I_{\ell-1},\bigcdot)$, so also $(J_1/I_{\ell-1},\bigcdot)\leq Z(B/I_{\ell-1},\bigcdot)$ and hence $J_1/I_{\ell-1}\leq\zeta(B/I_{\ell-1})$. The claim is proved.

\medskip

Now, repeated applications of the previous claim allow us to assume~\hbox{$m=0$} (compare with the proof of Theorem \ref{bellooo}); in particular, $|B/\zeta_\ell(B)|$ is a power of $2$. It follows from Theorem \ref{pbracecentrallynilp} that $B/\zeta_\ell(B)$ is centrally nilpotent, so $B$ is centrally nilpotent. Since the additive groups of the factors of the upper central series of $B$ are torsion-free and~$B/\zeta_n(B)$ is finite by hypothesis, we have that $B=\zeta_n(B)$, so $\Gamma_{n+1}(B)$ is finite.~\end{proof}

\medskip

Actually, the converse of the previous statement holds for every almost polycyclic brace. We need the following result which is deduced from \cite{tutti23-2}, Lemma 3.3 and Remark 3.5.

\begin{lem}\label{bellooo}
Let $B$ be a brace. If $N\leq M$ are ideals of $B$ such that \begin{itemize}
    \item[\textnormal{(a)}] $N$ is finite,
    \item[\textnormal{(b)}] $M/N\leq\zeta(B/N)$, and
    %\item[\textnormal{(c)}] $(M/N,+)$ is finitely generated,
\end{itemize}
then $M$ contains an ideal $L$ of $B$ such that: \begin{itemize}
    \item[\textnormal{(1)}] $L\cap N=\{0\}$;
    \item[\textnormal{(2)}] $L\leq\zeta(B)$;
    \item[\textnormal{(3)}] $M/L$ is finite;
   % \item[\textnormal{(4)}] $(L,+)$ is torsion-free.
\end{itemize}

\noindent In particular, $M/M\cap\zeta(B)$ is finite.
\end{lem}

\begin{theo}
Let $B$ be an almost polycyclic brace such that $\Gamma_{n+1}(B)$ is finite for some~\hbox{$n\in\mathbb{N}$.} Then $B/\zeta_n(B)$ is finite.
\end{theo}
\begin{proof}
By Lemma \ref{bellooo}, $\Gamma_n(B)$ contains a finite-index subbrace $I$ of $B$ such that $I\leq\zeta(B)$. By induction, $U/I=\zeta_{n-1}(B/I)$ has finite-index in $B$, so also $\zeta_n(B)\geq U$ has finite index in $B$.
\end{proof}

\begin{cor}
Let $B$ be a supersoluble brace. Then $\Gamma_{n+1}(B)$ is finite if and only if $B/\zeta_n(B)$ is finite.
\end{cor}

In the final part of this section, we show how to recognize supersolubility in an almost polycyclic brace. One of the  ways is to look at the finite homomorphic images.

\begin{theo}\label{finitehomsupers}
Let $B$ be an almost polycyclic brace. Then $B$ is supersoluble if and only if all its finite homomorphic images are supersoluble.
\end{theo}
\begin{proof}
If $B$ is supersoluble, then obviously such are also its (finite) homomorphic images. Conversely, suppose $B$ is infinite and that all its finite homomorphic images are supersoluble. Consider the natural semidirect product $G=N\rtimes_\lambda X$, where~\hbox{$N=(B,+)$} and $X=(B,\bigcdot)$. Then, $G$ is polycyclic-by-finite, since from the definition of almost polycyclic brace it follows that both $(B,+)$ and $(B,\bigcdot)$ are polycyclic-by-finite, and this class is closed with respect to forming extensions. If~$H$ is any finite-index normal subgroup of $G$, then there is a finite-index ideal $I$ of~$B$ such that $M=(I,+)\leq H\cap N$ and $Y=(I,\bigcdot)\leq H\cap X$ (see~\cite{tutti23-2},~The\-o\-rem~3.10). Since $G/MY\simeq (B/I,+)\rtimes_\lambda (B/I,\bigcdot)$, we have that $G/H$ is supersoluble by~Lem\-ma~\ref{supersolublebraceisgroup}. The arbitrariness of~$H$ shows that $G$ is supersoluble (see \cite{baer}, Satz 3.1).

Now, let $I$ be a non-trivial normal cyclic subgroup of $G$ contained in~\hbox{$(\operatorname{Soc}(B),+)\leq N$} (this is possible by~The\-o\-rem~\ref{exth3.4}). Then~Lem\-ma~1.10 of~\cite{Cedo} yields that $I\trianglelefteq B$. By induction, $B/I$ is supersoluble, and consequently $B$ is supersoluble.
\end{proof}

\medskip

The other way is through finite presentations of almost polycyclic braces.

\begin{theo}\label{decide}
There is an algorithm which, given a presentation of an almost polycyclic brace $B$, decides if~$B$ is supersoluble.
\end{theo}
\begin{proof}
Two recursive procedures are set in motion. The first constructs all finite braces in increasing orders and finds finitely many homomorphisms from $B$ into each of these braces. Then it tests if the image of $B$ is supersoluble. By Theorem \ref{finitehomsupers}, this procedure stops if $B$ is not supersoluble.

In order to complete the proof of the statement we need a procedure stopping if~$B$ is supersoluble. By Theorem 4.7 of \cite{tutti23-2}, we find a set $\{b_1,\ldots,b_n\}$ of additive and multiplicative generators of $B$, and, by Theorem 4.13 of \cite{tutti23-2}, we find~\hbox{$n\in\mathbb{N}$} and $I\trianglelefteq B$ such that $I\leq\operatorname{Soc}_n(B)$ and~$B/I$ is finite; in particular, we find a set $\{a_1,\ldots,a_m\}$ of additive and multiplicative generators of $I$. Suppose $I\neq\{0\}$ (see \cite{tutti23-2}, Theorem 4.4). Then we recursively enumerate all $b$-words with symbols $x_1,\ldots,x_m$, and for each of these $b$-words, say $\theta(x_1,\ldots,x_m)$, we  check if $[a,b_i]_+=a\ast b_i=a\ast a=0\neq a$ for all~\hbox{$1\leq i\leq n$,} where $a=\theta(a_1,\ldots,a_m)$ (see \cite{tutti23-2}, Theorem 4.1). If the result is negative, we move to the next $b$-word. If the result is positive, it means that $\langle a\rangle_+\leq\operatorname{Soc}(B)$; then we also check if $\langle a\rangle_+$ is an ideal of $B$ by using Theorems 4.4 and 4.9 of \cite{tutti23-2}. If so, we check if $I=\langle a\rangle_+$ (see \cite{tutti23-2}, Theorem 4.4); otherwise, we move to the next~\hbox{$b$-word.} If~\hbox{$I=\langle a\rangle_+$,} then $B/\langle a\rangle_+$ is finite; if not, we move to $B/\langle a\rangle_+$ and we repeat the previous procedure until we find, if possible, a finite chain of ideals of $B$ $$\{0\}=I_0\leq I_1\leq\ldots\leq I_\ell=I$$ such that, for each $0\leq i<\ell$, $I_{i+1}/I_i\leq\operatorname{Soc}(B/I_i)$ and $(I_{i+1}/I_i,+)$ is cyclic. Then we move to $B/I$, which is finite, so without loss of generality we may assume $B$ is finite. As before, we recursively enumerate all elements of $B$, and for each of them, say $x$, we check if $\langle x\rangle_+$ is an ideal of $B$ such that $\langle x\rangle_{\bigcdot}=\langle x\rangle_+$. If $x$ is such an element, then we check if $B=\langle x\rangle_+$. If so, the algorithm stops. If not, we repeat the same procedure to the quotient $B/\langle x\rangle_+$, until we find, if possible, a finite chain of ideals of $B$ $$\{0\}=J_0\leq J_1\leq\ldots\leq J_r=I$$ such that, for each $0\leq i<r$, $(I_{i+1}/I_i,+)$ and $(I_{i+1}/I_i,\bigcdot)$ are cyclic groups (see also~Re\-mark~\ref{remchepoiserve}). It is clear that if $B$ is supersoluble, then the algorithm stops after finitely many steps.~\end{proof}

%Moreover, we can find $u_1,\ldots,u_m\in B$ such that $(I,+)=\langle u_1,\ldots,u_m\rangle_+$. By Theorem \ref{wordproblem}, we can check if $I$ is zero or not. If $I=\{0\}$, then~$B$ is finite and we are easily done. Suppose $I\neq\{0\}$, so $\operatorname{Soc}(B)\neq\{0\}$. Let $b_1,\ldots,b_n$ be generators for $(B,+)$. By recursively enumerating the $b$-words in the values $b_1,\ldots,b_n$, we enumerate all elements of $B$, and we can then enumerate all elements of $\operatorname{Soc}(B)$ (we just need to check if for  such an element $c$ of $B$ we have $c\ast b_i=0=[c,b_i]_+$ for every $i$). If $B$ is supersoluble, there must be an element $d\in\operatorname{Soc}(B)$ such that $\langle d\rangle_+$ is an ideal of~$B$. The ideal generated by $d$ can be constructed (that is, we get additive generators of such an ideal) using Theorem \ref{stronleftnormalclosure} and~\cite{Cedo}, Lem\-ma~1.10; then one checks if this ideal coincides with $\langle d\rangle_+$ using Theorem \ref{alg1}. If this is an ideal, call it $J$, then we move to~$B/J$. Now, either $(B/J,+)$ has a smaller Hirsch length of $B$, or $(B,+)$ and $(B/J,+)$ have the same Hirsch length but the largest normal finite subgroup of $(B/J,+)$ has smaller order than that of $(B,+)$. Thus, induction on the lexicographically ordered pairs $(r,s)$, where $r$ is the Hirsch length of $(B,+)$, and $s$ is the order of the largest finite normal subgroup of $(B,+)$, shows that this procedure stops after finitely many steps.

\section{Hypercyclic and locally supersoluble  braces}\label{sectls}

The aim of this section is to provide a broader context to some of the results on supersoluble braces in a fashion which is similar to that in which some of the main results for central nilpotency are proved in (and generalized to) the broader contexts of locally centrally-nilpotent braces and hypercentral braces (see for example \cite{tutti23}). Here, the two key definitions are those of hypercyclic brace and locally supersoluble brace. A brace $B$ is {\it locally supersoluble} if its finitely generated subbraces are supersoluble, while it is {\it hypercyclic} if it has an ascending chain of ideals $$\{0\}=I_0\leq I_1\leq\ldots I_\alpha\leq I_{\alpha+1}\leq\ldots I_\lambda=B$$ \textnormal($\alpha$ and $\lambda$ are ordinal numbers\textnormal) such that, for each ordinal $\beta<\lambda$, either $(I_{\beta+1}/I_\beta,+)$ is an infinite cyclic group and $I_{\beta+1}/I_\beta\leq\operatorname{Soc}(B/I_\beta)$, or $I_{\beta+1}/I_\beta$ has prime order.

Clearly, every supersoluble brace is hypercyclic, and every hypercyclic brace is locally supersoluble by Corollary 3.3 of \cite{Tr23}, but the converses do not hold. Moreover, as for supersoluble braces (see Lemma \ref{lsupersolublebraceisgroup}), it turns out that all the relevant groups connected with a locally supersoluble brace have nice generalized supersoluble properties (we refer to \cite{Rob72} for the definitions and properties of hypercyclic and locally supersoluble groups).

\begin{lem}\label{lsupersolublebraceisgroup}
Let $B$ be a brace, and let $G=(B,+)\rtimes_\lambda(B,\bigcdot)$. If $B$ is locally supersoluble \textnormal(resp. hypercyclic\textnormal), then $G$  is supersoluble \textnormal(resp. hypercyclic\textnormal). In particular, $(B,+)$ and $(B,\bigcdot)$ are locally supersoluble \textnormal(resp. hypercyclic\textnormal) groups.
\end{lem}
\begin{proof}
We only prove the result for a hypercyclic brace $B$. By definition, $B$ has a non-zero ideal $I$ such that either $I\leq\operatorname{Soc}(B)$ and $(I,+)$ is infinite cyclic, or $|I|$ is a prime. Then $M=(I,+)$ and $N=(I,+)\rtimes_\lambda(I,\bigcdot)$ are normal subgroups of $G$ such that $M$ and~$N/M$ are cyclic groups, and $G/N\simeq (B/I,+)\rtimes_\lambda(B/I,\bigcdot)$. Iterating this procedure by transfinite recursion, we may define an ascending series of normal subgroups of $G$ with cyclic factors that must reach $G$ at some point.
\end{proof}

\medskip

In Section \ref{sectsuper}, we showed that a brace is supersoluble if and only if it is weakly supersoluble. This immediately implies that a brace is locally supersoluble if and only if it is locally weakly-supersoluble. Taking inspiration from the definition of a weakly supersoluble brace, we could define weakly hypercyclic braces in the obvious way (requiring the infinite factors to just be contained in the additive centre). Although we do not wish to bore the reader with a precise definition, we now show that again these concepts turn out to be equivalent (so there is actually no need to give a definition).

\begin{lem}\label{lemequivalenza2}
Let $B$ be a locally supersoluble brace and let $I$ be an ideal of $B$.
\begin{itemize}
    \item[\textnormal{(1)}] If $I\leq Z(B,+)$ and $|I|$ is a prime, then $I\leq\operatorname{Soc}(B)$.
    \item[\textnormal{(2)}] If $I$ is abelian and $(I,+)$ is infinite cyclic, then $I\leq\operatorname{Soc}(B)$.
\end{itemize}
\end{lem}
\begin{proof}
If $F$ is any finitely generated subbrace of $B$ containing $I$, then $I\leq\operatorname{Soc}(B)$ by~Lem\-ma~\ref{lemequivalenza}. The arbitrariness of $F$ yields $I\leq\operatorname{Soc}(B)$ and completes the proof.
\end{proof}

\begin{theo}
Let $B$ be a brace having an ascending chain of ideals $$\{0\}=I_0\leq I_1\leq\ldots I_\alpha\leq I_{\alpha+1}\leq\ldots I_\lambda=B$$ \textnormal(here $\alpha$ and $\lambda$ are ordinal numbers\textnormal) such that, for each ordinal $\beta<\lambda$, either $I_{\beta+1}/I_\beta$ is abelian and $(I_{\beta+1}/I_\beta,+)$ is a central cyclic subgroup of $(B/I_\beta,+)$, or $I_{\beta+1}/I_\beta$ has prime order. Then $B$ is hypercyclic.
\end{theo}
\begin{proof}
It follows from Corollary 3.3 of \cite{Tr23} that $B$ is locally weakly-supersoluble, so also locally supersoluble by Theorem \ref{weaksupersoluble}. Now, let $J\leq I$ be ideals of $B$ such that~$I/J$ is abelian and $(I/J,+)$ is a central cyclic subgroup of $(B/J,+)$. If~$I/J$ is infinite, then, since~$B/J$ is locally supersoluble, it follows from Lemma \ref{lemequivalenza2} that $I/J\leq\operatorname{Soc}(B/J)$. Suppose $I/J$ is finite. Since $(I/J,+)$ and $(I/J,\bigcdot)$ are finite cyclic groups, there is a finite chain of ideals of $B$ $$I=J_0\leq J_1\leq\ldots\leq J_\ell=J$$ whose factors have prime order (see the proof of Theorem \ref{squarefreeorder} for more details). Therefore $B$ is hypercyclic.
\end{proof}

\medskip

There are a couple of ways to detect the property of being hypercyclic starting from some sub-structural entities of a brace. For example, the following characterization of a hypercyclic brace (an analog of which can be given for hypercentral braces) shows that we need only look for small ideals in every quotient of the brace.

\begin{theo}\label{charhypercyclic}
Let $B$ be a brace. Then $B$ is hypercyclic if and only if every non-zero quotient~$B/I$ of $B$ has a non-zero ideal $J/I$ such that either $J/I\leq\operatorname{Soc}(B/I)$ and $(J/I,+)$ is an infinite cyclic group, or $|J/I|$ is prime.
\end{theo}

\medskip

By definition, local supersolubility can be detected from the behaviour of the finitely generated subbraces. It is not possible to state the same for the property of being hypercyclic, but if we replace ‘‘finitely generated’’ by ‘‘countable’’, then this can be achieved (note that every finitely generated brace is countable). In fact, it has been shown in \cite{tutti23} that the countable subbraces of a brace $B$ play a relevant role in determining which properties $B$ satisfies; in particular, Theorem 4.4 of \cite{tutti23} states that the property of being centrally nilpotent and that of being hypercentral can be both detected from the analysis of the countable subbraces. Our next result is a further example of this fact.

\begin{theo}
Let $B$ be a brace. Then $B$ is hypercyclic if and only if all its countable subbraces are hypercyclic.
\end{theo}
\begin{proof}
Of course, we only need to show that if all countable subbraces of $B$ are hypercyclic, then $B$ is hypercyclic. Moreover, by Theorem \ref{charhypercyclic} (see also the proof of~The\-o\-rem~\ref{squarefreeorder}), we only need to find a non-zero $1$-generator subgroup $I$ of $(B,+)$ that has the following property with respect to $B$: 
\begin{itemize}
    \item[($\star$)] $I$ is an abelian ideal of $B$, and either $|I|$ is finite or $I\leq\operatorname{Soc}(B)$.
\end{itemize}
Suppose there is no such a subgroup. Let $0\neq x\in B$ and put $X_0=\langle x\rangle$; in particular,~$X_0$ is a countable subbrace of $B$. Assume now we have defined a countable subbrace~$X_n$ of $B$ for some non-negative integer $n$. For each $y\in X_n\setminus\{0\}$, we have that $\langle y\rangle_+$ does not satisfy ($\star$), so either $y\ast y\neq 0$ (in this case, we let $b_y=0$), or there is an element~\hbox{$b_y\in B$} such that one of the following conditions holds:
\begin{itemize}
    \item $\langle y\rangle_+$ is not an ideal of~$\langle X_n,b_y\rangle$.
    \item $\langle y\rangle_+$ is an infinite ideal of $\langle X_n,b_y\rangle$ and $y\ast b_y\neq0$.
\end{itemize}
Define $X_{n+1}=\langle X_n,b_y\,:\, y\in X_n\setminus\{0\}\rangle$. Finally, put $$X=\bigcup_{n\geq0}X_n.$$ Since $X_i\leq X_{i+1}$ for all $i$, we have that $X$ is a countable subbrace of $B$. Thus, $(X,+)$ contains a non-zero $1$-generator subgroup $U=\langle a\rangle_+$ satisfying  ($\star$) with respect to~$X$. Therefore, there is $\ell\geq0$ such that $a\in X_\ell$, and this means that $U=\langle a\rangle_+$ does not satisfy property ($\star$) with respect to $\langle X_{\ell},b_a\rangle\leq X_{\ell+1}\leq X$, an obvious contradiction.
\end{proof}

\medskip

Now, we deal with some structural properties of locally supersoluble (resp. hypercyclic) braces. The first thing we note is that one can immediately generalize~Co\-rol\-la\-ry~\ref{psubbrace} to arbitrary locally supersoluble braces.

\begin{theo}\label{oddprime}
Let $B$ be a locally supersoluble brace and let $p$ be a prime number. Then $U_p^+(B)=U_p^{\bigcdot}(B)$ is an ideal of $B$.
\end{theo}
\begin{proof}
Let $x\in U_p^{+}(B)$, $b\in B$, and put $X=\langle x,b\rangle$. Then $$\langle x\rangle^X\leq U_p^+(X)=U_p^{\bigcdot}(X)\leq U_p^{\bigcdot}(B).$$ A symmetric argument shows that $\langle y\rangle^Y\leq U_p^+(B)$, where $y\in U_p^{\bigcdot}$ and $Y=\langle y,b\rangle$. This is clearly enough to prove the statement.
\end{proof}

\begin{rem}
{\rm Actually, the conclusions of Theorem \ref{oddprime} are satisfied by any brace that locally satisfy them (the proof being the same).}
\end{rem}

\medskip

As shown in \cite{5}, a great deal of structural information  is provided by the chief factors of a brace. Recall that if $B$ is brace and $J<I$ are ideals of $B$ such that $I/J$ is a minimal ideal of~$B/J$, then $I/J$ is said to be a {\it chief factor} of $B$. It is obvious from the definition that chief factors of supersoluble braces have prime order, and our next aim is to show that this is still true for locally supersoluble braces (see Corollary \ref{corchieflocallsuper}). As a by-product, we will also describe chief factors of locally soluble (see Corollary \ref{corchieffacls}) and locally polycyclic braces (see Corollary \ref{chieffactorLP}) --- recall that a brace $B$ is {\it polycyclic} if it has a finite chain of ideals $$\{0\}=I_0\leq I_1\leq\ldots\leq I_n=B$$ such that $I_{i+1}/I_i\leq\operatorname{Soc}(B/I_i)$ for all $0\leq i<n$ (see \cite{tutti23-2}). In order to do so, we need to introduce some further soluble-related concepts. Let $I$ be an ideal of a brace $B$. Then~$I$ is {\it $B$-hypoabelian} if there is a descending chain of ideals of $B$ $$\{0\}=I_\lambda\ldots\leq I_{\alpha+1}\leq I_\alpha\ldots \leq I_1\leq I_0=I$$ (here $\alpha$ and $\lambda$ are ordinal numbers) such that, for all $0\leq\beta<\lambda$, the factor $I_{\beta}/I_{\beta+1}$ is abelian; if $\lambda$ is a finite ordinal number, then $I$ is also said to be {\it $B$-soluble}. If $B$ is~\hbox{$B$-hypo}\-abelian, we also say that $B$ is {\it hypoabelian} and, clearly, $B$ is soluble if and only if it is~\hbox{$B$-soluble}. Note that hypoabelianity is not preserved with respect to forming quotients, and that every ideal of a hypoabelian (resp. soluble) brace $B$ is~\hbox{$B$-hypo}\-abe\-lian (resp. $B$-soluble).

\begin{theo}\label{chiefsoluble}
Let $I$ be a minimal ideal of a brace $B$. If, for every finitely generated subbrace~$F$ of $B$, the intersection $F\cap I$ is $F$-hypoabelian, then $I$ is an abelian brace which, as a group, can be described as either an elementary abelian $p$-group for some prime $p$, or a restricted direct product of copies of the additive group of the rational numbers.
\end{theo}
\begin{proof}
Suppose $I$ is non-abelian. Then there exist elements $a,b\in I$ such that \hbox{$S=\{[a,b]_+,a\ast b\}\neq\{0\}$.} Let $c\in S\setminus\{0\}$. The minimality of $I$ in $B$ implies that~\hbox{$I=\langle c\rangle^B$,} so there is a finite subset $E$ of $B$ such that $c\in E$ and $\langle a,b\rangle\leq H=\langle c\rangle^{K}$, where $K=\langle E\rangle$. Since $H\leq K\cap I$, we have that $H$ is $K$-hypoabelian and hence there is an ideal $L$ of $K$ which is strictly contained in $H$ and such that $H/L$ is abelian.  Thus, $c\in S\subseteq L$ and consequently $H=\langle c\rangle^K\leq L<H$, a contradiction. Therefore $I$ is abelian.

Now, let $p$ be any prime. Clearly, $I[p]=\{x\in I\,:\, x^p=0\}$ is an ideal of $B$, so either $I[p]=I$ or $I[p]=\{0\}$. Similarly, $I^{p,+}=I^{p,\bigcdot}$ is an ideal of $B$. This implies that either~$(I,+)$ is an elementary abelian $q$-group for some prime $q$, or it is torsion-free and divisible, which means it is a restricted direct product of copies of the additive group of rational numbers.
\end{proof}

\begin{cor}\label{corchieffacls}
Let $B$ be a locally soluble brace. Then every chief factor of $B$ is abelian and, as a group, can be described as either an elementary abelian $p$-group for some prime $p$, or a restricted direct product of copies of the additive group of the rational numbers.
\end{cor}
\begin{proof}
Since local solubility is preserved with respect to forming  quotients, we only need to prove the statement for a minimal ideal $I$ of $B$. Now, let~$F$ be any finitely generated subbrace of $B$. Then $F$ is soluble, so $F\cap I$ is $F$-soluble and hence~\hbox{$F$-hypo}\-abelian. An application of Theorem \ref{chiefsoluble} completes the proof.
\end{proof}

\begin{theo}\label{chiefpoly}
Let $I$ be a minimal ideal of a brace $B$. If for every finitely generated subbrace~$F$ of $B$ the intersection $F\cap I$ is finitely generated and $F$-hypoabelian, then $I$ is an abelian brace which, as a group, can be described as an elementary abelian $p$-group for some prime $p$.
\end{theo}
\begin{proof}
By Theorem \ref{chiefsoluble}, $I$ is abelian. Suppose by contradiction that $(I,+)$ is torsion-free, and let $0\neq x\in I$. Put $y=x^2$. The minimality of $I$ shows that $I=\langle y\rangle^B$. Thus, there is a finite subset~$E$ of $B$ such that $y\in E$ and $x\in \langle y\rangle^K$, where $K=\langle E\rangle$. By assumption,~\hbox{$K\cap I$} is finitely generated and abelian, so $(K\cap I,+)$ is finitely generated, and hence it is a direct product of finitely many infinite cyclic groups. Put $X=\langle x\rangle^K$; in particular, \hbox{$X\leq K\cap I$.} Now, $y=x^2\in X^{2,\bigcdot}=X^{2,+}\trianglelefteq K$, so $$x\in \langle y\rangle^K\leq X^{2,\bigcdot}\trianglelefteq K$$ and hence $X=X^{2,\bigcdot}$, a contradiction. Thus, $(I,+)$ is an elementary abelian $p$-group for some prime $p$ and we are done.
\end{proof}

\begin{cor}\label{chieffactorLP}
Let $B$ be a locally polycyclic brace. Then every chief factor of $B$ is abelian and, as a group, can be described as an elementary abelian~\hbox{$p$-group} for some prime $p$ \textnormal(depending on the chief factor\textnormal).
\end{cor}
\begin{proof}
Since the property of being locally polycyclic is preserved with respect to forming quotients, we only need to prove the statement for a minimal ideal $I$ of~$B$. Now, let~$F$ be any finitely generated subbrace of $B$. Then $F$ is polycyclic, so $F\cap I$ is~\hbox{$F$-hypo}\-abelian. Moreover, since every subbraces of a polycyclic brace is finitely generated, we have that~\hbox{$F\cap I$} is finitely generated. An application of~The\-o\-rem~\ref{chiefpoly} completes the proof.
\end{proof}

\medskip

In order to deal with the chief factors of locally supersoluble braces, we introduce the following definition. Let $n\in\mathbb{N}$. A brace $B$ is said to be {\it $n$-chief} if the following two conditions hold:
\begin{itemize}
    \item[(1)] $B$ satisfies the maximal condition on subbraces.
    \item[(2)] If $I/J$ is any chief factor of $B$, then $I/J$ is abelian and there is a prime $p$ (depending on $I/J$) such that~$(I/J,+)$ is an elementary abelian $p$-group of rank at most $n$ (in particular, $|I/J|\leq p^n$).
\end{itemize}

Of course, any $n$-chief brace is hypoabelian. Note also that although the property of being $n$-chief is inherited by forming quotients, it is not clear if it is inherited by subbraces. However, in our next result, we deal with locally~\hbox{$n$-chief} braces, a class of braces that is (by definition) closed with respect to forming subbraces, and is of course also closed with respect to forming quotients.

\begin{theo}\label{theosuperchiefpre}
Let $n\in\mathbb{N}$ and let $B$ be a locally $n$-chief brace. Then every chief factor of $B$ is abelian and, as a group, can be described as an elementary abelian~\hbox{$p$-group} of rank at most~$n$ for some prime~$p$ \textnormal(depending on the chief factor\textnormal).
\end{theo}
\begin{proof}
It is clearly enough to prove the statement for a minimal ideal $I$ of $B$. Let~$F$ be a non-zero, finitely generated subbrace of~$B$, and let $J$ be any non-zero ideal of~$F$. Since~$F$ satisfies the maximal condition on subbraces, there is an ideal $L$ of~$F$ that is maximal with respect to the property of being strictly contained in $J$. Thus~$J/L$ is a chief factor of~$F$ and is consequently abelian. The arbitrariness of $J$ in $F$ shows that~\hbox{$F\cap I$} is~\hbox{$F$-hypo}\-abe\-lian. Since $F\cap I$ is finitely generated, it follows from~The\-o\-rem~\ref{chiefpoly} that~$I$ is abelian, and that~$(I,+)$ is an elementary abelian $p$-group for some prime $p$. Suppose the rank of $I$ is at least $n+1$, and let $x_1,\ldots,x_{n+1}$ be additively independent elements of $I$. There are elements $b_1,\ldots,b_\ell$ of $B$ such that if $c\in\langle x_1,\ldots,x_{n+1}\rangle=V$, then $\langle c\rangle^U\geq V$, where $U=\langle x_1,\ldots,x_{n+1},b_1,\ldots,b_\ell\rangle$. Let $M$ be an ideal of $U$ which is maximal with respect to the property of not containing $x_1$. Then $V^UM/M$ is a non-zero chief factor of $U$ whose additive group is an elementary abelian $p$-group of rank at least $n+1$. This contradiction completes the proof.
\end{proof}

\begin{rem}\label{remnosubclosed}
{\rm Since we do not know if the property of being $n$-chief is subbrace-closed, a most reasonable definition for a local $n$-chief brace is that of a brace in which every finitely generated subbrace is contained in a subbrace that is $n$-chief. It is easy to see that the statement of Theorem \ref{theosuperchiefpre} still holds if we employ this definition, the proof being essentially the same.}
\end{rem}

Since every supersoluble brace is $1$-chief, the following result is an immediate consequence of Theorem \ref{theosuperchiefpre}.

\begin{cor}\label{corchieflocallsuper}
Let $B$ be a locally supersoluble brace. Then every chief factor of $B$ has prime order.
\end{cor}

\medskip

Another relevant structural piece of information concerns with the indices of the subbraces. In order to simplify the notation, we introduce the following definition. We say that a brace $B$ has the {\it index-property} if the following two assertions are equivalent for every subbrace $X$ of $B$:
\begin{itemize}
    \item[\textnormal{(1)}] $(X,+)$ has finite index $n$ in $(B,+)$.
    \item[\textnormal{(2)}] $(X,\bigcdot)$ has finite index $n$ in $(B,\bigcdot)$.
\end{itemize}

As we already remarked, every almost polycyclic brace (and consequently every supersoluble brace) has the index-property. As a consequence of the following more general result, it turns out that every locally supersoluble brace has the index-property.

\begin{theo}\label{locallyindexproperty}
Let $B$ be a brace that locally satisfies the index-property. Then $B$ has the index-property.
\end{theo}
\begin{proof}
Let $X$ be any subbrace of $B$ such that $|(B,+):(X,+)|$ is finite and is equal to~$n$. Let $b_1,\ldots,b_n$ be a  transversal to $(X,+)$ in $(B,+)$, and put $C=\langle b_1,\ldots,b_n\rangle$. Then $|(C,+):(C\cap X,+)|=n$. On the other hand, since $C$ is finitely generated, we have that $$|(C,+):(C\cap X,+)|=|(C,\bigcdot):(C\cap X,\bigcdot)|,$$ so $|(B,\bigcdot):(X,\bigcdot)|\geq  n$. If $|(B,\bigcdot):(X,\bigcdot)|>n$, then we can find $c_1,\ldots,c_{n+1}\in B$ such that $c_iX\neq c_jX$ for all $1\leq i,j\leq n+1$ with $i\neq j$. However, this means that $$|(D,+):(D\cap X,+)|=n<|(D,\bigcdot):(D\cap X,\bigcdot)|,$$ where $D=\langle c_1,\ldots,c_{n+1}\rangle$, a contradiction. Thus, $|(B,+):(X,+)|=|(B,\bigcdot):(X,\bigcdot)|$ and~$B$ has the index-property. We argue in a similar way in case $|(B,\bigcdot):(X,\bigcdot)|$ is finite.
\end{proof}

\begin{rem}
{\rm Similarly to Remark \ref{remnosubclosed}, we note that Theorem \ref{locallyindexproperty} holds if we replace the condition that the brace ‘‘locally satisfies the index-property’’ with the condition that ‘‘every finitely generated subbrace is contained in a subbrace satisfying the index-property’’.}
\end{rem}

\begin{cor}
Every brace that is locally almost polycyclic has the index-property. In particular, every locally supersoluble brace has the index-property.
\end{cor}

We could give many other analogs of the statements in Section \ref{sectsuper} concerning the index-structural properties of a supersoluble brace, but for the sake of conciseness we only concern ourselves with a couple of them. The first result generalizes~Co\-rol\-la\-ry~\ref{corsuperindex2}.

\begin{theo}
Let $B$ be a locally supersoluble brace. If $C$ is a subbrace of $B$ such that $|B:C|=2$, then $C\trianglelefteq B$.
\end{theo}
\begin{proof}
Suppose $C$ is not an ideal of $B$, and choose~\hbox{$b\in B$} and $c\in C$ such that one of the elements $\lambda_b(c)$, $c^{b,\bigcdot}$, $c^{b,+}$ is not contained in $C$. Now, the subbrace $D=\langle b,c\rangle$ is supersoluble and $C\cap D$ has index $2$ in $D$, so $C\cap D\trianglelefteq D$ by Co\-rol\-la\-ry~\ref{corsuperindex2}. However, $c\in C\cap D$ and $b\in D$, although one of the elements $\lambda_b(c)$, $c^{b,\bigcdot}$, $c^{b,+}$ does not belong to $C\cap D$. This contradiction completes the proof.
\end{proof}

\medskip

The other result we would like to generalize is Theorem \ref{maximalsubbrace}, which states that a maximal subbrace of a supersoluble brace has prime index. Unfortunately, we are not able to prove that every maximal subbrace of a locally supersoluble must have prime index, so we leave this as an open question --- as we shall see, the problem relies on the bad nilpotency-behaviour of the derived ideal of a supersoluble brace (see Example~\ref{badbehderideal} and Theorem~\ref{locallynilpointerse}). On the other hand, we can prove that at least in some special circumstances (such as the hypercyclic case) the result is true, but the proof is not a trivial consequence of Theorem \ref{maximalsubbrace}. It is in fact based on the relation between maximal subbraces and a locally nilpotent concept for ideals that has been introduced in \cite{tutti23}. Thus, recall first that an ideal $I$ of a brace $B$ is {\it locally $B$-nilpotent} if the following condition holds:
\begin{itemize}
    \item For every finitely generated subbrace $F$ of $B$, the finitely generated subbraces of~\hbox{$I\cap F$} are contained in $F$-centrally nilpotent ideals of~$F$.
\end{itemize}

Of course, every $B$-centrally nilpotent ideal is locally $B$-nilpotent. Moreover, every locally $B$-nilpotent ideal of a brace $B$ is locally nilpotent, and it follows from~The\-o\-rem~5.3 of \cite{tutti23} that the sum of arbitrarily many locally~\hbox{$B$-nil}\-potent ideals of a brace~$B$ is locally $B$-nilpotent. Thus,~$B$ has a largest locally~\hbox{$B$-nil}\-potent ideal, which we call the~{\it Hirsch--Plotkin ideal} of $B$. As for $B$-central nilpotency in supersoluble braces (see~The\-o\-rem~\ref{theobcentrally}), local $B$-nilpotency coincides with its ‘‘non-relative’’ analog in the locally supersoluble environment.

\begin{theo}\label{blocally}
Let $B$ be a locally supersoluble brace, and let $I$ be a locally centrally-nilpotent ideal of $B$. Then $I$ is locally $B$-nilpotent.
\end{theo}
\begin{proof}
Let $F$ be any finitely generated subbrace of $B$; in particular, $F$ is supersoluble and all the subbraces of $F$  are finitely generated. Being a finitely generated subbrace of the locally centrally-nilpotent ideal $I$, we have that $I\cap F$ is centrally nilpotent. Then~\hbox{$I\cap F$} is an $F$-centrally nilpotent ideal of $F$ by Theorem \ref{theobcentrally}. The arbitrariness of $F$ proves that $I$ is locally~\hbox{$B$-nil}\-po\-tent.~\end{proof}

\begin{cor}\label{locallycentrallynilp}
Let $B$ be a locally supersoluble brace. Then the join of all locally centrally-nilpotent ideals of $B$ is locally centrally-nilpotent and coincides with the Hirsch--Plotkin ideal of $B$.
\end{cor}

%As a consequence of Corollary \ref{cormaximalsuper}, we can easily prove by induction on the additive/multiplicative index that every locally finite-supersoluble brace $B$ with finite $\pi(B,+)$ has the index-property.

%The following result generalizes Corollary \ref{fittingideal} and shows that, in locally supersoluble braces, the {\it Hirsch--Plotkin radical} (that is, the join of all locally centrally-nilpotent ideals) is locally centrally-nilpotent. The proof follows by combining Lemma \ref{blocally} and~\cite{tutti23},~The\-o\-rem~5.23.

The key result that let us handle maximal subbraces of a hypercyclic brace is the following one. Recall that if $I$ is an ideal of a brace $B$, then $\partial_B(I)$ is the smallest ideal~$J$ of $B$ such that $J\leq I$ and $I/J$ is abelian --- clearly, $\partial_B(B)=\partial(B)$.

\begin{theo}\label{locallynilpointerse}
Let $B$ be a brace that locally satisfies the maximal condition on subbraces, and let~$I$ be a locally $B$-nilpotent ideal of $B$. If $M$ is any maximal subbrace of $B$, then \hbox{$\partial_B(I)\leq I\cap M\trianglelefteq B$.}
\end{theo}
\begin{proof}
Suppose $I\cap M$ is not an ideal of $B$; in particular, $I\not\leq M$. Since $M$ is maximal in $B$, we have that $B=IM$, so there are elements $x\in I$ and $a\in I\cap M$ such that $\mathcal{A}:=\big\{a^{x,\bigcdot},a^{x,+},\lambda_x(a)\big\}\not\subseteq M$. Let $b\in\mathcal{A}\setminus M$ and let $F$ be a finite subset of~$M$ such that $x\in\langle b,F\rangle$ (this subset exists because $M$ is maximal). Put $C=\langle b,a,F\rangle$ and let $D$ be a subbrace of $C$ that is maximal with respect to containing~\hbox{$\langle a,F\rangle$} but not $b$; in particular,~$D$ is a maximal subbrace of~$C$. Since $C$ satisfies the maximal condition on subbraces, we have that $I\cap C$ is a $C$-centrally nilpotent ideal of~$C$. Moreover,~\hbox{$a\in (I\cap C)\cap D$} and $x\in C$ but~\hbox{$b\not\in (I\cap C)\cap D$}, so without any loss of generality we may assume that~$I$ is $B$-centrally nilpotent.

Now, let $i$ be the largest non-negative integer such that $\zeta_i(I)_B\leq I\cap M$. Then $Z:=\zeta_{i+1}(I)_B\not\leq I\cap M$, so $B=ZM$. Since $I\cap M\trianglelefteq M$ and $$[z,I\cap M]_{\bigcdot}\cup [z,I\cap M]_+\cup\lambda_z(I\cap M)\subseteq I\cap M$$ for all $z\in Z$, we have that $I\cap M\trianglelefteq B$, thus completing the proof of the first part of the statement.

\medskip

Finally, suppose that $\partial_B(I)\not\leq I\cap M$. Since $I\cap M\trianglelefteq B$, we have $\partial(I)\not\leq I\cap M$. Let~\hbox{$a\in \partial(I)\setminus (I\cap M)$;} in particular, $a$ does not belong to $M$. Then there are finitely many elements $x_1,\ldots,x_n$ of $I$ such that $a\in \partial\big(\langle x_1,\ldots,x_n\rangle\big)$. For each $i\in\{1,\ldots,n\}$, let~$F_i$ be a finite subset of $M$ such that $x_i\in \langle a,F_i\rangle$. Put $C=\langle a,F_1,\ldots,F_n\rangle$. Since~\hbox{$x_i\in C$} for all $1\leq i\leq n$, we also have that $a\in\partial(I\cap C)$. Since $C$ satisfies the maximal condition on subbraces, we have that $I\cap C$ is a $C$-centrally nilpotent ideal of~$C$. Let $D$ be a subbrace of $C$ that is maximal with respect to containing~\hbox{$F_1,\ldots,F_n$} but not $a$. It follows that $D$ is actually a maximal subbraces of $C$, and that \hbox{$\partial(I\cap C)\not\leq (I\cap C)\cap D$.} Thus, without loss of generality we may assume $I$ is~\hbox{$B$-cen}\-tral\-ly nilpotent.
 Now, $B = \partial_B(I)M$, and hence
\[ I = I \cap \partial_B(I)M = \partial_B(I)(I \cap M).\] Since $I/(I\cap M)$ is non-zero and~\hbox{$\big(B/(I\cap M)\big)$-cen}\-tral\-ly nilpotent, we have that $$\partial_B(I)(I\cap M)/(I\cap M)\leq \partial_{B/(I\cap M)}\big(I/(I\cap M)\big)<I/(I\cap M),$$ a contradiction. Thus, $\partial_B(I)\leq I\cap M$ and the statement is proved.
\end{proof}

\begin{rem}
{\rm A statement similar to that of Theorem \ref{locallynilpointerse} also holds if the concept of subbrace is replaced by that of (strong) left-ideal; we leave the details to the reader but we just note that, for example, we need to replace the fact that the brace $B$ locally satisfies the maximal condition on subbraces by the request that every finite subset of~$B$ is contained in a (strong) left-ideal that satisfies the maximal condition on (strong) left-ideals. All these statements generalize Lemma 5.10 of~\cite{tutti23}.}
\end{rem}

Now, we delve into some of the main consequences of Theorem \ref{locallynilpointerse}. First, we note that if we define the {\it Frattini ideal} $\operatorname{Frat}(B)$ of a brace~$B$ as the intersection of all maximal subbraces of $B$, then the following analog of \cite{tutti23}, The\-o\-rem~5.12, holds (the proof being exactly the same).

\begin{theo}
Let $B$ be a finite brace such that $\operatorname{Fit}(B)\cap \operatorname{Frat}(B)=\{0\}$. Then $\operatorname{Fit}(B)$ is the product of all the minimal ideals of $B$ that are abelian. In particular, $\operatorname{Fit}(B)$ is abelian.
\end{theo}

\begin{cor}
Let $B$ be a finite brace. If $\operatorname{Frat}(B)=\{0\}$, then every $B$-centrally nilpotent ideal of~$B$ is abelian.
\end{cor}

A seemingly stronger result holds in the supersoluble case, by virtue of~The\-o\-rem~\ref{blocally} and the second half of~The\-o\-rem~\ref{locallynilpointerse}.

\begin{cor}
Let $B$ be a \textnormal(locally\textnormal) supersoluble brace. If $\operatorname{Frat}(B)=\{0\}$, then every \hbox{\textnormal(locally\textnormal)} centrally nilpotent ideal of~$B$ is abelian.
\end{cor}

Now, we deal with the problem of the index of a maximal subbrace. The first result in this context concerns with those braces that are {\it locally finite-supersoluble}, that is, those braces in which every finite subset is contained in a finite supersoluble subbrace.

\begin{cor}%\label{cormaximalsuper}
Let $B$ be a locally finite-supersoluble brace such that $\pi(B,+)$ is finite. If $M$ is any maximal subbrace of $B$, then $|B:M|$ is a prime.
\end{cor}
\begin{proof}
Let $V=M_B$ be the largest ideal of $B$ contained in $M$ (note that $B/V$ is not necessarily finite) and let $p$ be the largest prime in $\pi(B/V,+)$. It follows from The\-o\-rem~\ref{oddprime} that the set $P/V$ of all $p$-elements  of $(B/V,+)$ is an ideal of $B/V$. Moreover, since $P/V$ is locally finite, it follows from The\-o\-rem~\ref{pbracecentrallynilp} that $P/V$ is locally centrally-nilpotent. Then~Lem\-ma~\ref{blocally} shows that $P/V$ is locally $(B/V)$-nilpotent, so $P\cap M$ is an ideal of $B$ by~The\-o\-rem~\ref{locallynilpointerse}, and hence $P\cap M=V$. Now, $M/V$ is maximal in $B/V$, so $P/V$ does not contain any proper non-zero ideal of $B/V$, and hence $P/V$ is a chief factor of $B$.~Corollary~\ref{corchieflocallsuper} yields that $P/V$ has order $p$. Therefore $|B:M|=p$ and the statement is proved.
\end{proof}

\medskip

The other result concern with hypercyclic braces.

\begin{cor}%\label{indexmaximalsubbracehypercyc}
Let $B$ be a hypercyclic brace. If $M$ is a maximal subbrace of $B$, then $|B:M|$ is a prime number.
\end{cor}
\begin{proof}
Let $V=M_B$ be the largest ideal of $B$ contained in $M$, and let $U/V$ be a non-zero ideal of $B/V$ such that either $U/V\leq\operatorname{Soc}(B/V)$ and $U/V$ is $1$-generator, or~$|U/V|$ is a prime. Clearly, $U/V$ is locally $(B/V)$-nilpotent, so Theorem \ref{locallynilpointerse} yields that $M\cap U\trianglelefteq B$ and hence $M\cap U=V$. It follows that $U/V$ is a chief factor of $B/V$, so $|U/V|$ is a prime number, $p$ say, by Corollary \ref{corchieflocallsuper}.  Therefore $|B:M|=p$ and the statement is proved.
\end{proof}

\medskip

In the final part of the paper we deal more in detail with nilpotency properties of hypercyclic and locally supersoluble  braces. Most of the results here generalize corresponding results for supersoluble braces we saw in Section \ref{sectsuper}; for example, the following one generalizes~The\-o\-rem~\ref{schurtheor}.

\begin{theo}\label{lsschur}
Let $B$ be a locally supersoluble brace such that $B/\zeta_n(B)$ is locally finite for some $n\in\mathbb{N}$. Then $\Gamma_{n+1}(B)$ is locally finite.
\end{theo}
\begin{proof}
Let $\mathcal{F}$ be the set of all finitely generated subbraces of $B$, and let $F\in\mathcal{F}$. Then~$F$ is supersoluble and so $F/\zeta_n(F)$ is finite. Now, Theorem \ref{schurtheor} yields that $\Gamma_{n+1}(F)$ is finite. On the other hand, $$\Gamma_{n+1}(B)=\bigcup_{F\in\mathcal{F}}\Gamma_{n+1}(F),$$ and $\Gamma_{n+1}(F)\leq\Gamma_{n+1}(E)$ whenever $F\leq E$ are in $\mathcal{F}$. Thus, every finitely generated subbrace of $\Gamma_{n+1}(B)$ is contained in some $\Gamma_{n+1}(E)$, $E\in\mathcal{F}$, and is consequently finite. Therefore $\Gamma_{n+1}(B)$ is locally finite and the statement is proved.
\end{proof}

\begin{rem}
{\rm Note that the converse of Theorem \ref{lsschur} does not hold even for groups.}
\end{rem}

In Theorem \ref{rnilsupersoluble}, we characterized supersoluble braces of finite multipermutational level in terms of the nilpotency of their additive group. In the more general contexts of hypercyclic and locally supersoluble braces, something similar can be achieved. In what follows, a brace $B$ is {\it locally hypersocle} if for every finitely generated subbrace $F$ of $B$, we have $F=\overline{\operatorname{Soc}}(F)$. Examples of locally hypersocle braces are those braces that are (locally of) finite multipermutational level. 

\begin{rem}
{\rm We do not know if every locally hypersocle brace need to be locally of finite multipermutational level. This would certainly be the case if one could prove that a finitely generated brace $B$ with $B=\overline{\operatorname{Soc}}(B)$ is of finite multipermutational level. We actually believe this is not the case, but if it were, the condition ‘‘locally hypersocle’’ could  be rephrased simply as  ‘‘locally of finite multipermutational level’’.}
\end{rem}

\begin{theo}\label{lsupersofinimulti}
Let $B$ be a locally supersoluble brace. Then the following conditions are equivalent:
\begin{itemize}
    \item[\textnormal{(1)}] $(B,+)$ is locally nilpotent.
    \item[\textnormal{(2)}] $B$ is locally hypersocle.
    \item[\textnormal{(3)}] $B$ is locally of finite multipermutational level.
\end{itemize}
\end{theo}
\begin{proof}
If $(B,+)$ is locally nilpotent, then $B$ is locally of finite multipermutational level by~The\-o\-rem~\ref{rnilsupersoluble}, and so even locally hypersocle. Assume conversely that $B$ is locally hypersocle, and let $F$ be any finitely generated subbrace of $B$. Then $F$ is supersoluble, so it satisfies the maximal condition on subbraces, and consequently $F$ is of  finite multipermutational level. In particular, $(F,+)$ is nilpotent. The arbitrariness of $F$ shows that $(B,+)$ is locally nilpotent.
\end{proof}

\begin{theo}\label{incrediblenotcited}
Let $B$ be a hypercyclic brace. The following conditions are equivalent:
\begin{itemize}
    \item[\textnormal{(1)}] $(B,+)$ is hypercentral.
    \item[\textnormal{(2)}] $(B,+)$ is locally nilpotent.
    \item[\textnormal{(3)}] $B=\overline{\operatorname{Soc}}(B)$.
    \item[\textnormal{(4)}] $B$ is locally hypersocle.
    \item[\textnormal{(5)}] $B$ is locally of finite multipermutational level.
\end{itemize}
\end{theo}
\begin{proof}
Since $B$ is locally supersoluble, conditions (4) and (5) are clearly equivalent. Moreover, since $(B,+)$ is hypercyclic, conditions (1) and (2) are known to be equivalent. Moreover, (2) and (4) are equivalent by Theorem \ref{lsupersofinimulti}, because every hypercyclic brace is locally supersoluble. Finally, since clearly (3) implies (1), we only need to show that (1) implies (3). Assume therefore that $(B,+)$ is hypercentral, and let $I$ be any non-zero ideal of $B$ such that either $I\leq\operatorname{Soc}(B)$ and $(I,+)$ is infinite cyclic, or $|I|$ is a prime (this ideal exists by definition of hypercyclic brace). By Theorem \ref{charhypercyclic}, we only need to prove that $I\leq\operatorname{Soc}(B)$, so we may assume that $I$ has prime order. If~$F$ is any finitely generated subbrace of~$B$ containing $I$, then $I$ is a minimal ideal of $F$. Since $F$ is supersoluble, then $(F,+)$ is finitely generated and so also nilpotent. Thus, $F=\operatorname{Soc}_{n}(F)$ for some $n\in\mathbb{N}$ by~The\-o\-rem~\ref{rnilsupersoluble} and consequently $I\leq\operatorname{Soc}(F)$. The arbitrariness of $F$ shows that $I\leq\operatorname{Soc}(B)$ and completes the proof of the statement.~\end{proof}

\medskip

The following lemma enables us to prove ‘‘converses’’ of the above results.

\begin{lem}\label{lemunpostrange}
Let $B$ a locally hypersocle brace. The following statements are equivalent:
\begin{itemize}
    \item[\textnormal{(1)}] $B$ is finitely generated, and $(B,\bigcdot)$ locally satisfies the maximal condition on subgroups.
    \item[\textnormal{(2)}] $(B,\bigcdot)$ satisfies the maximal condition on subgroups.
    \item[\textnormal{(3)}] $(B,+)$ is finitely generated.
\end{itemize}
\end{lem}
\begin{proof}
Of course, we may assume $B=\overline{\operatorname{Soc}}(B)$. We first deal with the case $B=\operatorname{Soc}_n(B)$ for some $n\in\mathbb{N}$. It is clear that if either $(B,\bigcdot)$ or $(B,+)$ are finitely generated, then also $B$ is finitely generated; in particular, (2)$\implies$(1) and also (3)$\implies$(1). 
Moreover, if the nilpotent group is $(B,+)$ is finitely generated, then it satisfies the maximal condition on subgroups, so the multiplicative group of every abelian factor $\operatorname{Soc}_{i+1}(B)/\operatorname{Soc}_i(B)$ is finitely generated; consequently, $(B,\bigcdot)$ satisfies the maximal condition on subgroups and (3)$\implies$(2). Similarly, we have that (2)$\implies$(3).

Finally, assume that $B$ is finitely generated and that $(B,\bigcdot)$ locally satisfies the maximal condition on subgroups. By induction on $n$, we have that~$\big(B/\operatorname{Soc}(B),\bigcdot\big)$ is finitely generated. Thus, since $\big(B/\operatorname{Soc}(B),\bigcdot\big)$ satisfies the maximal condition on subgroups, we may employ a combination of Lemma 3.1 and Theorem 3.2 of \cite{Tr23} to shows that~$\operatorname{Soc}(B)$ is finitely generated as an ideal of $B$ by certain elements $x_1,\ldots,x_\ell$. Now, let $b_1,\ldots,b_m$ be elements of $B$ generating~$(B,\bigcdot)$ modulo $\operatorname{Soc}(B)$, and let~$S$ be the normal closure of~\hbox{$\langle x_1,\ldots,x_\ell\rangle_{\bigcdot}$} in $(B,\bigcdot)$. Then $S\leq\operatorname{Soc}(B)$, so~$S$ is also a normal subgroup of~$(B,+)$, and, by Lemma 1.10 of \cite{Cedo}, we have that $S$ is~\hbox{$\lambda$-in}\-va\-riant. Thus, $S\trianglelefteq B$ and consequently $S=\operatorname{Soc}(B)$. This clearly implies that $(B,\bigcdot)=\langle x_1,\ldots,x_\ell,b_1,\ldots,b_m\rangle_{\bigcdot}$ is finitely generated. Since $(B,\bigcdot)$ locally satisfies the maximal condition on subgroups, we have that (1)$\implies$(2), thus completing the proof in the special case $B=\operatorname{Soc}_n(B)$.

\medskip

Now, we turn to the general situation. Suppose that one of the conditions (1), (2), or~(3) holds, and let $\mu$ be the smallest ordinal number such that $B=\operatorname{Soc}_\mu(B)$. We assume by way of contradiction that $\mu\geq\omega$. Since, $B$, $(B,\bigcdot)$, or $(B,+)$ are finitely generated, it turns out that $\mu$ is a successor. Let $\lambda$ be the largest limit ordinal that is strictly less than $\mu$. By the first part of the proof, we have that $B/\operatorname{Soc}_\lambda(B)$ satisfies the maximal condition on subbraces (and also on additive/multiplicative subgroups). Then $\operatorname{Soc}_\lambda(B)$ is finitely generated either as an ideal of $B$ (see also \cite{Tr23}, Lemma 3.1), a normal subgroup of $(B,+)$, or a normal subgroup of $(B,\bigcdot)$. Since $\lambda$ is limit, we have that $\lambda=0$, contradicting the fact that~\hbox{$\lambda\geq\omega$.}
\end{proof}

\begin{theo}\label{allafineserve}
Let $B$ be a locally hypersocle brace. If $(B,\bigcdot)$ is locally supersoluble, then $B$ is locally supersoluble.
\end{theo}
\begin{proof}
Let $F$ be a finitely generated subbrace of $B$. Then $(F,\bigcdot)$ satisfies the maximal condition on subgroups by Lemma \ref{lemunpostrange}, so $F=\operatorname{Soc}_n(F)$ for some $n\in\mathbb{N}$. It follows from Theorem \ref{supersolublegroup} that $F$ is supersoluble. Then $B$ is locally supersoluble and the statement is proved.
\end{proof}

\begin{theo}\label{hypercyclicgroup}
Let $B$ be a brace such that $B=\overline{\operatorname{Soc}}(B)$. If $(B,\bigcdot)$ is hypercyclic, then $B$ is hypercyclic.
\end{theo}
\begin{proof}
Put $S=\operatorname{Soc}(B)$. Then $(S,\bigcdot)$ is a non-trivial normal subgroup of $(B,\bigcdot)$, so~$S$ has a non-trivial cyclic normal subgroup, $M$, say. Of course, $M$ is a strong left-ideal of~$B$, and an application of Lemma 1.10 of \cite{Cedo} shows that $M\trianglelefteq B$. Since the hypotheses are inherited by quotients, we have thus shown that every non-zero homomorphic image of $B$ has a non-zero ideal that is contained in the socle and is generated (as a subbrace) by only one element. Therefore, $B$ is hypercyclic by~The\-o\-rem~\ref{charhypercyclic}.
\end{proof}

\medskip

The above results show that in the special case of a hypercyclic (resp. locally supersoluble) brace of hypercentral (resp. locally nilpotent) type, the structure of the multiplicative group plays a major role. For example, we can prove the following analogs of Theorem \ref{baersuperso} (1).

\begin{theo}%\label{prodlocallysupero}
Let $B$ be a locally hypersocle brace, and let $I$ and $J$ be locally supersoluble ideals of $B$ such that $B=I+J$. If $[I,J]_{\bigcdot}$ is locally nilpotent, then $B$ is locally supersoluble.
\end{theo}
\begin{proof}
By~Lem\-ma~\ref{lsupersolublebraceisgroup}, the multiplicative groups of $I$ and $J$ are locally supersoluble. Moreover, $[B,B]_{\bigcdot}=[I,I]_{\bigcdot}[J,J]_{\bigcdot}[I,J]_{\bigcdot}$ is locally nilpotent. Then Lemma 2.2 of \cite{infiniteminimal} shows that $(B,\bigcdot)$ is locally supersoluble. Finally, an application of Theorem \ref{allafineserve} yields that $B$ is locally supersoluble.
\end{proof}

\begin{theo}
Let $B$ be a brace such that $B=\overline{\operatorname{Soc}}(B)$, and let $I$ and $J$ be hypercyclic ideals of~$B$ such that $B=I+J$. If $[I,J]_{\bigcdot}$ is locally nilpotent, then $B$ is hypercyclic.
\end{theo}
\begin{proof}
Of course, $[B,B]_{\bigcdot}=[I,I]_{\bigcdot}[J,J]_{\bigcdot}[I,J]_{\bigcdot}$ is locally nilpotent. By Lemma \ref{lsupersolublebraceisgroup}, $(B,\bigcdot)$ is the product of the two hypercyclic groups $(I,\bigcdot)$ and $(J,\bigcdot)$, so it is hypercyclic by~Co\-rol\-la\-ry~2.3 of \cite{infiniteminimal}. Then Theorem \ref{hypercyclicgroup} shows that $B$ is hypercyclic and completes the proof.~\end{proof}

\medskip

The last thing we remark about right-nilpotency of locally supersoluble braces concerns with the derived ideal and generalizes Theorem \ref{derivedisrightnilp}.

\begin{theo}
Let $B$ be a brace. 
\begin{itemize}
    \item[\textnormal{(1)}] If $B$ is locally supersoluble, then $\partial(B)$ is locally of finite multipermutational level.
    \item[\textnormal{(2)}] If $B$ is hypercyclic, then $\partial(B)=\overline{\operatorname{Soc}}\big(\partial(B)\big)$. 
\end{itemize}
\end{theo}
\begin{proof}
In order to prove (1), apply Theorem \ref{derivedisrightnilp} (and recall that every finite subset of $\partial(B)$ is contained in the derived ideal of some finitely generated subbrace of $B$). In order to prove (2), apply~The\-o\-rem~\ref{incrediblenotcited}.~\end{proof}

\medskip

We now turn to left-nilpotency. In this respect, we start proving the following analogs of Theorem \ref{leftnilptheo}, the first of which is not as obvious as one may believe. Recall that the hypercentre $\overline{Z}(G)$ of a group $G$ is the last term of the upper central series of $G$.

\begin{theo}\label{theoprec}
Let $B$ be a locally supersoluble brace that is locally left-nilpotent. If either
\begin{itemize}
    \item $\operatorname{Ker}(\lambda)\leq\overline{Z}(B,\bigcdot)$, or
    \item $\operatorname{Ker}(\lambda)\leq\overline{\operatorname{Soc}}(B)$,
\end{itemize}
then $(B,\bigcdot)$ is locally nilpotent.
\end{theo}
\begin{proof}
Let $F$ be any finitely generated subbrace of $B$; so in particular $F$ is supersoluble. It is enough to prove that~$(F,\bigcdot)$ is nilpotent. Since $F$ is an almost polycyclic brace, it follows from Theorem \ref{exth3.4} that there is $n\in\mathbb{N}$ such that $F/\operatorname{Soc}_n(F)$ is finite. On the other hand, $F$ is~\hbox{left-nil}\-po\-tent, so Lemma \ref{3.33} yields that $F/\zeta_m(B)$ is finite for some positive integer $m$. Then~The\-o\-rem~\ref{schurtheor} shows that $F/I$ is centrally nilpotent for some finite ideal $I$ of $F$. In particular,~$(F/I,\bigcdot)$ is nilpotent, and so, if $I\leq K=\operatorname{Ker}(\lambda)$, then~$(F,\bigcdot)$ is hypercentral (by Lemma \ref{3.33}) and even nilpotent by the maximal condition on subgroups.

Now, let $a$ be any element of $I$ that is not contained in $K$. Then there is a finitely generated subbrace $E_a$ of $B$ such that $F\leq E_a$ and $\lambda_a(b)\neq0$ for some $b\in E_a$. Using the argument we employed in the proof of Theorem \ref{leftnilptheo}, we see that $\big(E_a/R_a,\bigcdot\big)$ is nilpotent, where $R_a=\operatorname{Ker}(\lambda|_{E_a})$. Then $\big(F/(I\cap R_a),\bigcdot\big)$ is nilpotent and obvious\-ly~\hbox{$a\not\in I\cap R_a$.} It follows that $R=I\cap\bigcap_{a\in I\setminus K} R_a$ is a normal subgroup of~$(F,\bigcdot)$ such that $(F/R,\bigcdot)$ is nilpotent. Clearly, $R\leq K\cap F\leq\overline{Z}(F,\bigcdot)=Z_\ell(F,\bigcdot)$ for some $\ell\in\mathbb{N}$, since~$(F,\bigcdot)$ satisfies the maximal condition on subgroups. Therefore $(F,\bigcdot)$ is nilpotent and the statement is proved.
\end{proof}

\medskip

Also, direct applications of Corollary \ref{3.35} and \cite{tutti23-2}, Theorem 3.30, make it possible to prove the following result.

\begin{theo}
Let $B$ be a locally supersoluble brace of locally nilpotent type. The following conditions are equivalent:
\begin{itemize}
    \item[\textnormal{(1)}] $B$ is locally left-nilpotent.
    \item[\textnormal{(2)}] $(B,\bigcdot)$ is locally nilpotent.
    \item[\textnormal{(3)}] $B$ is locally centrally-nilpotent.
\end{itemize}
\end{theo}

In the hypercyclic context, we can prove the following result.

\begin{theo}\label{4.40}
Let $B$ be a hypercyclic brace that is locally left-nilpotent.
If either
\begin{itemize}
    \item $\operatorname{Ker}(\lambda)\leq\overline{Z}(B,\bigcdot)$, or
    \item $\operatorname{Ker}(\lambda)\leq\overline{\operatorname{Soc}}(B)$,
\end{itemize}
then $(B,\bigcdot)$ is hypercentral.
\end{theo}
\begin{proof}
It follows from Theorem \ref{theoprec} that $(B,\bigcdot)$ is locally nilpotent. On the other hand, since $B$ is hypercyclic, we have that $(B,\bigcdot)$ is hypercyclic by Lemma \ref{lsupersolublebraceisgroup}. Therefore,~$(B,\bigcdot)$ is hypercentral.
\end{proof}

\begin{cor}
Let $B$ be a hypercyclic brace of locally nilpotent type. Then $B$ is locally left-nilpotent if and only if $B$ is hypercentral.
\end{cor}
\begin{proof}
This follows from Lem\-ma~\ref{3.33}, Theorem \ref{incrediblenotcited} and Theorem~\ref{4.40}.
\end{proof}

\medskip

We have already scatteredly dealt with (local) central nilpotency and hypercentrality in the previous results. This was unavoidable due to the many connections they have with all kinds of properties of a brace. But there is more to be said. For example, we saw in Theorem \ref{blocally} that for a locally supersoluble brace $B$, the concepts of locally centrally-nilpotent ideal and locally $B$-nilpotent ideal coincide. Of course, this also holds for hypercyclic braces, but in this latter case we can add something more.

\begin{theo}\label{bhypercentral}
Let $B$ be a hypercyclic brace, and let $I$ be a \textnormal(non-zero\textnormal) locally centrally-nilpotent ideal of $B$. Then $I$ is $B$-hypercentral.
\end{theo}
\begin{proof}
Since $B$ is hypercyclic, we can find a non-zero ideal $J$ of $B$ such that
\begin{itemize}
    \item $J\leq I$, and
    \item either $J\leq\operatorname{Soc}(B)$ and $(J,+)$ is an infinite cyclic group, or $|J|$ is prime.
\end{itemize}
We claim that $J\!\leq\!\zeta(I)$. Since $I$ is locally centrally-nilpotent, we have that \hbox{$L/M\leq\zeta(I/M)$} whenever $L/M$ is a chief factor of $I$ (see \cite{tutti23}, Theorem 4.6). Thus, if $|J|$ is prime, then~\hbox{$J\leq\zeta(I)$.} On the other hand, if $J$ is infinite, then $J/J^{p,+}\leq\zeta(B/J^{p,+})$ for every prime~$p$, so clearly $J\leq\zeta(B)$ also in this case. This shows that $\zeta_1(I)_B\neq\{0\}$. Finally, an easy transfinite induction shows that $I=\zeta_\alpha(I)_B$ for some ordinal $\alpha$, and completes the proof of the statement.
\end{proof}

\begin{cor}%\label{corhypercentralb}
Let $B$ be a hypercyclic brace. If $I$ is a hypercentral ideal of $B$, then $I$ is~\hbox{$B$-hyper}\-central.
\end{cor}

\begin{cor}
Let $B$ be a hypercyclic brace. If $B$ is locally a finite brace of prime power order, then $B$ is hypercentral.
\end{cor}
\begin{proof}
This follows from Theorem \ref{bhypercentral} and  Theorem \ref{pbracecentrallynilp}.
\end{proof}

\medskip

It has been proved in \cite{tutti23} that the join of finitely many $B$-hypercentral ideals of a brace~$B$ is still $B$-hypercentral. As a consequence of the above results, we can say that in a hypercyclic brace, the join of arbitrarily many hypercentral ideals is hypercentral.

\begin{cor}
Let $B$ be a hypercyclic brace. Then the join of all hypercentral ideals of $B$ is hypercentral, and coincides with the Hirsch--Plotkin ideal of $B$.
\end{cor}
\begin{proof}
This follows from Corollary \ref{locallycentrallynilp}, Theorem \ref{bhypercentral}, and the fact that every hypercentral brace is locally centrally-nilpotent.
\end{proof}

\medskip

In case of a locally supersoluble brace $B$ whose additive group is torsion-free, we can say something more about the quotient of $B$ by the Hirsch--Plotkin ideal.

\begin{theo}
Let $B$ be a locally supersoluble brace, and let $H$ be the Hirsch--Plotkin ideal of $B$. If $U_2^+(B)=\{0\}$, then $U_2^+(B/H)=\{0\}$.
\end{theo}
\begin{proof}
Suppose $U$ is not locally centrally-nilpotent, where $U/H:=U_2^+(B/H)$. Then there is a finitely generated subbrace $F$ of $U$ that is not locally centrally-nilpotent. Consequently~$|F/F\cap H|$ is odd, so Theorem \ref{fittingpowerof2} yields that $F$ is centrally nilpotent, a contradiction.
\end{proof}

\medskip

If a locally supersoluble brace $B$ contains a large $B$-hypercentral ideal, then the brace is bound to be hypercyclic.

\begin{theo}%\label{lshypercentralhypercyclic}
Let $B$ be a locally supersoluble brace, and let $I$ be a $B$-hypercentral ideal of $B$ such that $B/I$ is finitely generated. Then $B$ is hypercyclic.
\end{theo}
\begin{proof}
Let $Z=\zeta(I)_B$, and let $x_1,\ldots,x_n$ generating $B$ modulo $I$. If~\hbox{$a\in Z$}, then $C=\langle a,x_1,\ldots,x_n\rangle$ is supersoluble, which means that $\langle a\rangle^C\leq I$ contains a non-zero ideal $J$ of~$C$ such that either $J\leq\operatorname{Soc}(C)$ and $(J,+)$ infinite cyclic, or $|J|$ is a prime. Since~\hbox{$B=IC$,} we have that $J\trianglelefteq B$. Then the result follows by Theorem \ref{charhypercyclic}.
\end{proof}

\medskip

We end this paper by noting the following analog of Lemma \ref{torsionfreesupersocatena} (that could be of an independent interest), the proof being the same.

\begin{theo}
Let $B$ be a hypercyclic brace, and let $I$ be a hypercentral ideal such that $(I,+)$ is torsion-free. Then there is an ascending chain $$\{0\}=I_0\leq I_1\leq\ldots I_\alpha\leq I_{\alpha+1}\leq\ldots I_\lambda=I$$ of ideals of $B$ such that, for all $\beta<\lambda$, $I_{\beta+1}/I_\beta\leq\operatorname{Soc}(B/I_\beta)\cap\zeta(I/I_\beta)$ and $(I_{\beta+1}/I_\beta,+)$ is infinite cyclic.
\end{theo}

\section{Acknowledgements}

We wish to thank the referee for their useful comments and remarks that really improved the paper.

\begin{flushleft}
\rule{8cm}{0.4pt}\\
\end{flushleft}

{
\sloppy
\noindent
Adolfo Ballester-Bolinches, Ramon Esteban-Romero, Vicent P\'erez-Calabuig

\noindent
Departament de Matem\`atiques

\noindent
Universitat de Val\`encia, Dr.\ Moliner, 50, 46100 Burjassot, Val\`encia (Spain)

\noindent
Adolfo.Ballester@uv.es;\;\; Ramon.Esteban@uv.es;\;\; Vicent.Perez-Calabuig@uv.es

}

\bigskip
\bigskip

{
\sloppy
\noindent
Maria Ferrara

\noindent
Dipartimento di Matematica e Fisica

\noindent
Università degli Studi della Campania  ``Luigi Vanvitelli''

\noindent
viale Lincoln 5, Caserta (Italy)

\noindent
e-mail: maria.ferrara1@unicampania.it
}

\bigskip
\bigskip

{
\sloppy
\noindent
Marco Trombetti

\noindent 
Dipartimento di Matematica e Applicazioni ``Renato Caccioppoli''

\noindent
Università degli Studi di Napoli Federico II

\noindent
Complesso Universitario Monte S. Angelo

\noindent
Via Cintia, Napoli (Italy)

\noindent
e-mail: marco.trombetti@unina.it 

}

\end{document}